\documentclass[11pt]{amsart}

\usepackage[top=23mm, bottom=23mm, left=26mm, right=26mm]{geometry}
\usepackage{color}
\usepackage{graphicx}
\usepackage{bm}
\usepackage{array}
\usepackage{kbordermatrix}

\newcommand{\vct}[1]{\bm{\mathsf{#1}}}

\newcommand{\mtx}[1]{\bm{\mathsf{#1}}}

\newtheorem{thm}{Theorem}[section]

\newtheorem{prop}[thm]{Proposition}

\newtheorem{remark}{Remark}
\theoremstyle{definition}

\newcommand{\pgnotate}[1]{}

\newcommand{\Cmm}{C_{\rm mm}}
\newcommand{\Cqr}{C_{\rm qr}}

\begin{document}

\begin{center}

\textsc{\large{A randomized blocked algorithm for efficiently computing rank-revealing factorizations of matrices}}

\vspace{3mm}

\textit{
Per-Gunnar Martinsson and Sergey Voronin \\
Department of Applied Mathematics, University of Colorado at Boulder}

\vspace{3mm}

\begin{minipage}{130mm} 
\textbf{Abstract:}
This manuscript describes a technique for computing
partial rank-revealing factorizations, such as, e.g, a partial
QR factorization or a partial singular value decomposition. The method
takes as input a tolerance $\varepsilon$ and an $m\times n$ matrix
$\mtx{A}$, and returns an approximate low rank factorization of $\mtx{A}$ that is
accurate to within precision $\varepsilon$ in the Frobenius norm (or some other
easily computed norm).
The rank $k$ of the computed factorization (which is an output of the algorithm)
is in all examples we examined very close to the theoretically optimal $\varepsilon$-rank.
The proposed method is inspired by the Gram-Schmidt algorithm, and has the same
$O(mnk)$ asymptotic flop count. However, the method relies on randomized
sampling to avoid column pivoting, which allows it to be blocked, and hence
accelerates practical computations by reducing communication.
Numerical experiments demonstrate that the accuracy of the scheme is
for every matrix that was tried at least as good as column-pivoted QR, and is sometimes
much better. Computational speed is also improved substantially, in particular
on GPU architectures.
\end{minipage}
\end{center}

\section{Introduction}

\subsection{Problem formulation} This manuscript describes an algorithm based on
randomized sampling for computing an approximate low-rank factorization of a
given matrix. To be precise, given a real or complex matrix $\mtx{A}$ of size $m\times n$,
and a computational tolerance $\varepsilon$, we seek to determine a matrix $\mtx{A}_{\rm approx}$
of low rank such that
\begin{equation}
\|\mtx{A} - \mtx{A}_{\rm approx}\| \leq \varepsilon.
\end{equation}
For any given $k \in \{1,\,2,\,\dots,\,\min(m,n)\}$, a rank-$k$ approximation to $\mtx{A}$
that is in many ways optimal
is given by the partial singular value decomposition (SVD),
\begin{equation}
\label{eq:svdk}
\begin{array}{ccccccccccccccccccccc}
\mtx{A}_{k} &=& \mtx{U}_{k} & \mtx{\Sigma}_{k} & \mtx{V}_{k}^{*},\\
m\times n && m\times k & k \times k & k\times n
\end{array}
\end{equation}
where $\mtx{U}_{k}$ and $\mtx{V}_{k}$ are orthonormal matrices whose
columns consist of the first $k$ left and right singular vectors, respectively,
and $\mtx{\Sigma}_{k}$ is a diagonal matrix whose diagonal entries
$\{\sigma_{j}\}_{j=1}^{k}$ are the leading $k$ singular values of $\mtx{A}$,
ordered so that $\sigma_{1} \geq \sigma_{2} \geq \sigma_{3} \geq \cdots \geq \sigma_{k} \geq 0$.
The Eckart-Young theorem \cite{1936_eckart_young} states that for the spectral norm and the Frobenius
norm, the residual error is minimal,
$$
\|\mtx{A} - \mtx{A}_{k}\| = \inf\{\|\mtx{A} - \mtx{C}\|\,\colon\,\mtx{C}\ \mbox{has rank}\ k\}.
$$
However, computing the factors in (\ref{eq:svdk}) is computationally expensive.
In contrast, our objective is to find an approximant $\mtx{A}_{\rm approx}$ that is cheap
to compute, and \textit{close} to optimal.

The method we present is designed for situations where $\mtx{A}$ is sufficiently large
that computing the full SVD is not economical. The method is designed to be highly
communication efficient, and to execute efficiently on both shared and distributed
memory machines. It has been tested numerically for situations where the matrix
fits in RAM on a single machine. We will, without loss of generality, assume that
$m \geq n$. For the most part we discuss real matrices, but the generalization to
complex matrices is straight-forward.

\subsection{A greedy template}
\label{sec:greedy}
A standard approach in computing low-rank factorizations is to employ a greedy
algorithm to build, one vector at a time, an orthonormal basis $\{\vct{q}_{j}\}_{j=1}^{k}$
that approximately spans the columns of $\mtx{A}$. To be precise, given an $m\times n$
matrix $\mtx{A}$ and a computational tolerance $\varepsilon$, our objective is to
determine a rank $k$, and an $m\times k$ matrix $\mtx{Q}_{k} = [\vct{q}_{1} \cdots \vct{q}_{k}]$ with
orthonormal column vectors such that $\|\mtx{A} - \mtx{Q}_k \mtx{B}_{k}\| \leq \varepsilon$,
where $\mtx{B}_{k} = \mtx{Q}^*_k \mtx{A}$. The matrices $\mtx{Q}_k$ and $\mtx{B}_k$
may be constructed jointly via the following procedure:

\begin{center}
\sffamily
\fbox{
\begin{minipage}{70mm}
\begin{tabbing}
\hspace{10mm} \= \hspace{5mm} \= \hspace{5mm} \= \hspace{5mm}\kill
Algorithm 1  \\[2mm]
(1) \> $\mtx{Q}_{0} = [\ ]$; $\mtx{B}_{0} = [\ ]$; $\mtx{A}_{0} = \mtx{A}$; $j=0$; \\[1mm]
(2) \> \textbf{while} $\|\mtx{A}^{(j)}\| > \varepsilon$\\[1mm]
(3) \> \> $j = j+1$\\
(4) \> \> Pick a unit vector $\vct{q}_{j} \in \mbox{ran}(\mtx{A}^{(j-1)})$.\\[1mm]
(5) \> \> $\vct{b}_{j}   = \vct{q}_{j}^{*}\mtx{A}^{(j-1)}$\\[1mm]
(6) \> \> $\mtx{Q}_{j}   = [\mtx{Q}_{j-1}\ \vct{q}_{j}]$\\[1mm]
(7) \> \> $\mtx{B}_{j}   = \left[\begin{array}{c} \mtx{B}_{j-1} \\ \vct{b}_{j} \end{array}\right]$\\[1mm]
(8) \> \> $\mtx{A}^{(j)} = \mtx{A}^{(j-1)} - \vct{q}_{j}\vct{b}_{j}$\\[1mm]
(9) \> \textbf{end while}\\[1mm]
(10)\> $k=j$.
\end{tabbing}
\end{minipage}}
\end{center}

Note that $\mtx{A}^{(j)}$ can overwrite $\mtx{A}^{(j-1)}$.
It can be verified
that if the algorithm is executed in exact
arithmetic, then the matrices generated satisfy
\begin{equation}
\label{eq:QB_alg_relations1}
\mtx{A}^{(j)} = \mtx{A} - \mtx{Q}_{j}\mtx{Q}_{j}^{*}\mtx{A},
\qquad\mbox{and}\qquad
\mtx{B}_{j} = \mtx{Q}_{j}^{*}\mtx{A}.
\end{equation}
The performance of the greedy scheme is determined by how we choose
the vector $\vct{q}_{j}$ on line (4). If we pick $\vct{q}_{j}$ as
simply the largest column of $\mtx{A}^{(j-1)}$, scaled to yield a
vector of unit length, then we recognize the scheme as the column
pivoted Gram-Schmidt algorithm for computing a QR factorization.
This method often works very well, but can lead to
sub-optimal factorizations. Reference \cite{gu1996} discusses
this in detail, and also provides an improved pivoting technique
that can be proved to yield closer to optimal results. However,
both standard Gram-Schmidt (see, e.g., \cite[Sect. 5.2]{MR3024913}),
and the improved version in \cite{gu1996} are challenging to implement
efficiently on modern multicore processors since they cannot readily
be blocked. Expressed differently, they rely on BLAS2 operations rather
than BLAS3.

Another natural choice for $\vct{q}_{j}$ on line (4) is to pick
the unit vector that minimizes $\|\mtx{A}^{(j-1)} - \vct{q}\vct{q}^{*}\mtx{A}^{(j-1)}\|$.
This in fact leads to an optimal factorization, with the vectors $\{\vct{q}_{j}\}_{j=1}^{k}$
being left singular vectors of $\mtx{A}$. However, finding the
minimizer tends to be computationally expensive.

In this manuscript, we propose a scheme that is more computationally efficient than
column-pivoted Gram-Schmidt, and often yields close to minimal approximation errors.
The idea is to choose $\vct{q}_{j}$ as
a random linear combination of the columns of $\mtx{A}^{(j-1)}$.
To be precise, we propose the following mechanism for choosing $\vct{q}_{j}$:
\begin{center}
\sffamily
\fbox{
\begin{minipage}{70mm}
\begin{tabbing}
\hspace{10mm} \= \hspace{5mm} \= \hspace{5mm} \= \hspace{5mm}\kill
(4a) \> Draw a random vector $\vct{\omega}$ whose entries are iid Gaussian random variables.\\[1mm]
(4b) \> Set $\vct{y} = \mtx{A}^{(j-1)}\vct{\omega}$.\\[1mm]
(4c) \> Normalize so that $\vct{q}_{j} = \frac{1}{\|\vct{y}\|}\,\vct{y}$.
\end{tabbing}
\end{minipage}}
\end{center}
This scheme is mathematically very close to the low-rank approximation scheme
proposed in \cite{2011_martinsson_randomsurvey}, but is slightly different
in the stopping criterion used (the scheme of \cite{2011_martinsson_randomsurvey}
does not explicitly update the matrix, and therefore relies on a probabilistic
stopping criterion), and in its performance when executed with finite precision
arithmetic.
We argue that choosing the vector $\vct{q}_{j}$ using randomized sampling
leads to performance very comparable to traditional column pivoting, but
has a decisive advantage in that the resulting algorithm is easy to block.
We will demonstrate substantial practical speed-up on both multicore CPUs
and GPUs.

\begin{remark}
The factorization scheme described in this section
produces an approximate factorization of the form $\mtx{A} \approx \mtx{Q}_{k}\mtx{B}_{k}$,
where $\mtx{Q}_{k}$ is orthonormal, but no conditions are \`a priori imposed
on $\mtx{B}_{k}$. Once the factors $\mtx{Q}_{k}$ and $\mtx{B}_{k}$ are available, it is
simple to compute many standard factorizations such as the low rank QR, SVD, or CUR factorizations.
For details, see Section \ref{sec:standard_factorizations}.
\end{remark}

\section{Technical preliminaries}

\subsection{Notation}
\label{sec:notation}
Throughout the paper, we measure vectors in $\mathbb{R}^{n}$ using their
Euclidean norm. The default norm for matrices will be the Frobenius norm
$\|\mtx{A}\| = \left(\sum_{i,j}|\mtx{A}(i,j)|^{2}\right)^{1/2}$, although
other norms will also be discussed.

We use the notation of Golub and Van Loan \cite{MR3024913} to specify
submatrices. In other words, if $\mtx{B}$ is an $m\times n$ matrix with
entries $b_{ij}$, and $I = [i_{1},\,i_{2},\,\dots,\,i_{k}]$ and $J =
[j_{1},\,j_{2},\,\dots,\,j_{\ell}]$ are two index vectors, then we
let $\mtx{B}(I,J)$ denote the $k\times \ell$ matrix
$$
\mtx{B}(I,J) = \left[\begin{array}{cccc}
b_{i_{1}j_{1}} & b_{i_{1}j_{2}} & \cdots & b_{i_{1}j_{\ell}} \\
b_{i_{2}j_{1}} & b_{i_{2}j_{2}} & \cdots & b_{i_{2}j_{\ell}} \\
\vdots         & \vdots         &        & \vdots         \\
b_{i_{k}j_{1}} & b_{i_{k}j_{2}} & \cdots & b_{i_{k}j_{\ell}}
\end{array}\right].
$$
We let $\mtx{B}(I,:)$ denote the matrix $\mtx{B}(I,[1,\,2,\,\dots,\,n])$,
and define $\mtx{B}(:,J)$ analogously.

The transpose of $\mtx{B}$ is denoted $\mtx{B}^{*}$, and we say that a matrix $\mtx{U}$
is \textit{orthonormal} if its  columns form an orthonormal set, so that $\mtx{U}^{*}\mtx{U} = I$.

\subsection{The singular value decomposition (SVD)}
\label{sec:SVD}
The SVD was introduced briefly in the introduction. Here we define it again, with some
more detail added. Let $\mtx{A}$ denote an $m\times n$ matrix, and set $r = \min(m,n)$.
Then $\mtx{A}$ admits a factorization
\begin{equation}
\label{eq:svd}
\begin{array}{ccccccccccccccccccccc}
\mtx{A} &=& \mtx{U} & \mtx{\Sigma} & \mtx{V}^{*},\\
m\times n && m\times r & r \times r & r\times n
\end{array}
\end{equation}
where the matrices $\mtx{U}$ and $\mtx{V}$ are orthonormal, and $\mtx{\Sigma}$ is diagonal.
We let $\{\vct{u}_{i}\}_{i=1}^{r}$ and $\{\vct{v}_{i}\}_{i=1}^{r}$ denote the columns of
$\mtx{U}$ and $\mtx{V}$, respectively. These vectors are the left and right singular vectors
of $\mtx{A}$. As in the introduction, the diagonal elements $\{\sigma_{j}\}_{j=1}^{r}$ of
$\mtx{\Sigma}$ are the singular values of $\mtx{A}$. We order these so that
$\sigma_{1}  \geq \sigma_{2} \geq \cdots \geq \sigma_{r} \geq 0$.
We let $\mtx{A}_{k}$ denote the truncation of the SVD to its first $k$ terms, as defined
by (\ref{eq:svdk}). It is easily verified that
\begin{equation}
\label{eq:minerrors}
\|\mtx{A} - \mtx{A}_{k}\|_{\rm spectral} = \sigma_{k+1},
\qquad\mbox{and that}\qquad
\|\mtx{A} - \mtx{A}_{k}\| = \left(\sum_{j=k+1}^{\min(m,n)} \sigma_{j}^{2}\right)^{1/2},
\end{equation}
where $\|\mtx{A}\|_{\rm spectral}$ denotes the operator norm of $\mtx{A}$ and
$\|\mtx{A}\|$ denotes the Frobenius norm of $\mtx{A}$. Moreover,
the Eckart-Young theorem \cite{1936_eckart_young} states that these errors are the smallest possible
errors that can be incurred when approximating $\mtx{A}$ by a matrix of rank $k$.


\subsection{The QR factorization}
\label{sec:QR}
Any $m\times n$ matrix $\mtx{A}$ admits a \textit{QR factorization} of the form
\begin{equation}
\label{eq:QR}
\begin{array}{ccccccccccccccccccccc}
\mtx{A} & \mtx{P} &=& \mtx{Q} & \mtx{R},\\
m\times n & n \times n && m\times r & r \times n
\end{array}
\end{equation}
where $r = \min(m,n)$, $\mtx{Q}$ is orthonormal, $\mtx{R}$ is upper triangular,
and $\mtx{P}$ is a permutation matrix. The permutation matrix $\mtx{P}$ can more efficiently
be represented via a vector $J_c \in \mathbb{Z}_{+}^{n}$ of column indices
such that $\mtx{P} = \mtx{I}(:,J_c)$ where $\mtx{I}$ is the $n\times n$ identity matrix.
Then (\ref{eq:QR}) can be written
\begin{equation}
\label{eq:QR2}
\begin{array}{ccccccccccccccccccccc}
\mtx{A}(\colon,J_c) &=& \mtx{Q} & \mtx{R},\\
m\times n && m\times r & r \times n
\end{array}
\end{equation}
The QR-factorization is often computed via column pivoting combined
with either the Gram-Schmidt process, Householder reflectors,
or Givens rotations \cite{MR3024913}. The resulting upper triangular $\mtx{R}$
then satisfies various decay conditions \cite{MR3024913}.
These techniques are all incremental, and can be stopped after the first $k$
terms have been computed to obtain a ``partial QR-factorization of $\mtx{A}$'':
\begin{equation}
\label{eq:QR3}
\begin{array}{ccccccccccccccccccccc}
\mtx{A}(\colon,J_c) &\approx& \mtx{Q}_{k} & \mtx{R}_{k}.\\
m\times n && m\times r & r \times n
\end{array}
\end{equation}
The main drawback of the
classical partial pivoted QR approximation is the difficulty to obtain substantial speedups
on multi-processor architectures.

\subsection{Orthonormalization}
\label{sec:orth}
Given an $m\times \ell$ matrix $\mtx{X}$, with $m \geq \ell$, we introduce the function
$$
\mtx{Q} = \texttt{orth}(\mtx{X})
$$
to denote orthonormalization of the columns of $\mtx{X}$. In other words, $\mtx{Q}$ will be
an $m\times \ell$ orthonormal matrix whose columns form a basis for the column space of $\mtx{X}$.
In practice, this step is typically achieved most efficiently by a call to a packaged QR
factorization (e.g., in Matlab, we would write $\mtx{Q} = \texttt{qr}(\mtx{X},0)$).
This step could in principle be implemented without pivoting, which makes this call
efficient.

\section{Construction of low-rank approximations via randomized sampling}
\label{sec:random}

\subsection{A basic randomized scheme}
\label{sec:randQB}
Let $\mtx{A}$ be a given $m\times n$ matrix whose singular values exhibit some
decay, and suppose that we seek a matrix $\mtx{Q}$ with $\ell$ orthonormal
columns such that
\begin{equation}
\label{eq:QBbasic}
\mtx{A} \approx \mtx{Q}\,\mtx{B},
\qquad\mbox{where}\qquad \mtx{B} = \mtx{Q}^{*}\,\mtx{A}.
\end{equation}
In other words, we seek a matrix $\mtx{Q}$ whose columns form an approximate
orthonormal basis for the column space of $\mtx{A}$.
A randomized procedure for solving this task was proposed in
\cite{2006_martinsson_random1_orig}, and later analyzed an elaborated in
\cite{2011_martinsson_random1,2011_martinsson_randomsurvey}.
A basic version of the scheme that we call ``\texttt{randQB}''
is given in Figure \ref{fig:randQB}.
Once \texttt{randQB} has been executed to produce the factors $\mtx{Q}$ and $\mtx{B}$
in (\ref{eq:QBbasic}), standard factorizations such as the QR factorization, or the
truncated SVD can easily be obtained, as described in Section
\ref{sec:standard_factorizations}.

\begin{figure}
\fbox{
\begin{minipage}{75mm}

\begin{tabbing}
\hspace{10mm} \= \hspace{60mm} \= \kill
     \> \textbf{function} $[\mtx{Q},\mtx{B}] = \texttt{randQB}(\mtx{A},\ell)$ \\[3mm]
(1)  \> $\mtx{\Omega} = \texttt{randn}(n,\ell)$ \\[1mm]
(2)  \> $\mtx{Q} = \texttt{orth}(\mtx{A}\mtx{\Omega})$\> {\color{blue}$\Cmm\,mn\ell + \Cqr\,m\ell^2$}\\[1mm]
(3)  \> $\mtx{B} = \mtx{Q}^{*}\mtx{A}$ \> {\color{blue}$\Cmm\,mn\ell$}
\end{tabbing}
\end{minipage}}
\caption{The most basic version of the randomized range finder. The algorithm
takes as input an $m\times n$ matrix $\mtx{A}$ and a target rank $\ell$, and produces
factors $\mtx{Q}$ and $\mtx{B}$ of sizes $m\times \ell$ and $\ell\times n$, respectively,
such that $\mtx{A} \approx \mtx{Q}\mtx{B}$.
Text in blue refers to computational cost, see Section \ref{sec:execution1} for notation.}
\label{fig:randQB}
\end{figure}

\subsection{Over sampling and theoretical performance guarantees}
\label{sec:theory}
The algorithm \texttt{randQB} described in Section \ref{sec:randQB} produces close
to optimal results for matrices whose singular values decay rapidly, provided that
some slight over-sampling is done. To be precise, if we seek to match the minimal
error for a factorization of rank $k$, then choose $\ell$ in \texttt{randQB} as
$$
\ell = k + s
$$
where $s$ is a small integer (say $s=10$).
It was shown in \cite[Thm.~10.5]{2011_martinsson_randomsurvey} that if $s \geq 2$, then
$$
\mathbb{E}\bigl[\|\mtx{A} - \mtx{Q}\mtx{B}\|\bigr]
\leq
\left(1 + \frac{k}{s-1}\right)\,
\left(\sum_{j=k+1}^{\min(m,n)}\sigma_{j}^{2}\right)^{1/2},
$$
where $\mathbb{E}$ denotes expectation. Recall from equation (\ref{eq:minerrors}) that
$\left(\sum_{j=k+1}^{\min(m,n)}\sigma_{j}^{2}\right)^{1/2}$ is the theoretically minimal
error in approximating $\mtx{A}$ by a matrix of rank $k$, so we miss the optimal bound
only by a factor of $\left(1 + \frac{k}{s-1}\right)$ (except for the over sampling, of course).
Moreover, it can be shown that the likelihood
of a substantial deviation from the expectation is extremely small \cite[Sec.~10.3]{2011_martinsson_randomsurvey}.

\begin{remark}
When errors are measured in the \textit{spectral norm}, as opposed to the Frobenius norm,
the randomized scheme is slightly further removed from optimality. Theorem 10.6
of \cite{2011_martinsson_randomsurvey} states that
\begin{equation}
\label{eq:spectralerror}
\mathbb{E}\bigl[\|\mtx{A} - \mtx{Q}\mtx{B}\|_{\rm spectral}\bigr]
\leq
\left(1 + \frac{k}{s-1}\right)\,\sigma_{k+1}
+
\frac{e \sqrt{k+s}}{s}\,
\left(\sum_{j=k+1}^{\min(m,n)}\sigma_{j}^{2}\right)^{1/2},
\end{equation}
where $e = 2.718\cdots$ is the basis of the natural exponent. We observe that in cases
where the singular values decay slowly, the right hand side of (\ref{eq:spectralerror})
is substantially larger than the theoretically optimal value of $\sigma_{k+1}$. For
such situation, the ``power scheme'' described in Section \ref{sec:RSVDpower} should be used.
\end{remark}

\subsection{Computing standard factorizations}
\label{sec:standard_factorizations}
The output of the randomized factorization scheme in Figure \ref{fig:randQB} is a factorization
$\mtx{A} \approx \mtx{Q}\mtx{B}$ where $\mtx{Q}$ is orthonormal, but no constraints have been
placed on $\mtx{B}$. It turns out that standard factorizations can efficiently be computed
from the factors $\mtx{Q}$ and $\mtx{B}$; in this section we describe how to get the QR, the
SVD, and ``interpolatory'' factorizations.

\subsubsection{Computing the low rank SVD}
\label{sec:rSVD}
To get a low rank SVD, cf.~Section \ref{sec:SVD}, we perform the full SVD on the $\ell \times n$ matrix $\mtx{B}$,
to obtain a factorization $\mtx{B} = \hat{\mtx{U}}\,\hat{\mtx{D}}\,\hat{\mtx{V}}$. Then,
\begin{equation*}
\mtx{A} \approx \mtx{Q} \mtx{B} = \mtx{Q} \hat{\mtx{U}} \hat{\mtx{D}} \hat{\mtx{V}}^*.
\end{equation*}
We can now choose a rank $k$ to use based on the decaying singular values of $\mtx{D}$. Once a
suitable rank has been chosen, we form the low rank SVD factors:
\begin{equation*}
\mtx{U}_k = \mtx{Q} \hat{\mtx{U}}(:,1:k),\qquad
\mtx{\Sigma}_k = \hat{\mtx{D}}(1:k,1:k),\qquad\mbox{and}\qquad
\mtx{V}_k = \hat{\mtx{V}}(:,1:k),
\end{equation*}
so that $\mtx{A} \approx \mtx{U}_k \mtx{\Sigma}_k \mtx{V}^*_k$.
Observe that the truncation undoes the over-sampling that was done and detects a
numerical rank $k$ that is typically very close to the optimal $\varepsilon$-rank.

\subsubsection{Computing the partial pivoted QR factorization}
To obtain the factorization $\mtx{A} \mtx{P} \approx \mtx{Q} \mtx{R}$, cf.~Section \ref{sec:QR},
from the QB decomposition, perform a QR factorization of the $\ell \times n$ matrix $\mtx{B}$ to
obtain $\mtx{B} \mtx{P} = \tilde{\mtx{Q}} \mtx{R}$. Then, set $\hat{\mtx{Q}} = \mtx{Q}\tilde{\mtx{Q}}$
to obtain
\begin{equation*}
\mtx{A} \mtx{P} \approx \mtx{Q} \mtx{B} \mtx{P} = \mtx{Q} \tilde{\mtx{Q}} \mtx{R} = \hat{\mtx{Q}} \mtx{R}.
\end{equation*}

\subsubsection{Computing interpolatory and CUR factorizations}
In applications such as data interpretation, it is often of interest to determine a
subset of the rows/columns of $\mtx{A}$ that form a good basis for its row/column space.
For concreteness, suppose that $\mtx{A}$ is an $m\times n$ matrix of rank $k$, and
that we seek to determine an index set $J$ of length $k$, and a matrix $\mtx{Y}$ of
size $k\times n$ such that
\begin{equation}
\label{eq:colID}
\begin{array}{ccccccccccccc}
\mtx{A} &\approx& \mtx{A}(:,J)&\mtx{Y}.\\
m\times n && m\times k & k\times n
\end{array}
\end{equation}
One can prove that there always exist such a factorization for which every entry of
$\mtx{Y}$ is bounded in modulus by $1$ (which is to say that the columns in
$\mtx{A}(:,J)$ form a well-conditioned basis for the range of $\mtx{A}$) and for
which $\mtx{Y}(:,J)$ is the $k\times k$ identity matrix \cite{2005_martinsson_skel}.
Now suppose that we have available a factorization $\mtx{A} = \mtx{Q}\mtx{B}$ where $\mtx{B}$
is of size $\ell\times n$. Then determine $J$ and $\mtx{Y}$ such that
\begin{equation}
\begin{array}{ccccccccccccc}
\mtx{B} &\approx& \mtx{B}(:,J)&\mtx{Y}.\\
\ell\times n && \ell\times k & k\times n
\end{array}
\end{equation}
This can be done using the techniques in, e.g., \cite{2005_martinsson_skel} or \cite{gu1996}.
Then (\ref{eq:colID}) holds \textit{automatically} for the index set $J$ and the matrix $\mtx{Y}$
that were constructed.
Using similar ideas, one can determine a set of rows that form a well-conditioned basis for
the row space, and also the so called CUR factorization
$$
\begin{array}{ccccccccccccc}
\mtx{A} &\approx& \mtx{C}&\mtx{U}&\mtx{R}^{*},\\
m\times n && m\times k & k\times k & k\times n
\end{array}
$$
where $\mtx{C}$ and $\mtx{R}$ consist of subsets of the columns and rows of $\mtx{A}$, respectively, cf.~\cite{2014arXivCURID}.

\section{A blocked version of the randomized range finder}
\label{sec:blocked}

In this section, we describe and analyze a \textit{blocked} version of the basic randomized scheme
described in Figure \ref{fig:randQB}. By blocking, we improve computational efficiency and simplify
implementation on parallel machines. Blocking also allows for \textit{adaptive rank determination}
to be incorporated for situations where the rank is not known in advance.
While blocking greatly helps with computational efficiency, it also creates some issues in terms of
aggregation of round-off errors; this problem can be eliminated using techniques described in
Section \ref{sec:roundoff}.

The algorithm described in this section is directly inspired by
Algorithm 4.2 of \cite{2011_martinsson_randomsurvey}; besides blocking, the scheme proposed
here is different in that the matrix $\mtx{A}$ is updated in a manner analogous to
``modified'' column-pivoted Gram-Schmidt. This updating allows the randomized stopping
criterion employed in \cite{2011_martinsson_randomsurvey} to be replaced with a
precise deterministic stopping criterion.

\subsection{Blocking}
\label{sec:basicblocking}
Converting the basic scheme in Figure \ref{fig:randQB} to a blocked scheme is in principle
straight-forward. Suppose that in addition to an $m\times n$ matrix $\mtx{A}$ and a rank
$\ell$, we have set a block size $b$ such that $\ell = sb$, for some integer $s$.
Then draw an $n\times \ell$ Gaussian random matrix $\mtx{\Omega}$, and partition it into
slices $\{\mtx{\Omega}_{j}\}_{j=1}^{s}$, each of size $n\times b$, so that
\begin{equation}
\label{eq:OmegaBlock}
\mtx{\Omega} = \bigl[\mtx{\Omega}_{1},\,\mtx{\Omega}_{2},\,\dots,\,\mtx{\Omega}_{s}\bigr].
\end{equation}
We analogously partition the matrices $\mtx{Q}$ and $\mtx{B}$ in groups of $b$ columns and $b$ rows, respectively,
$$
\mtx{Q} = \bigl[\mtx{Q}_{1},\,\mtx{Q}_{2},\,\dots,\,\mtx{Q}_{s}\bigr]
\qquad\mbox{and}\qquad
\mtx{B} = \left[\begin{array}{c} \mtx{B}_{1} \\ \mtx{B}_{2} \\ \vdots \\ \mtx{B}_{s}\end{array}\right].
$$
The blocked algorithm then proceeds to build the matrices $\{\mtx{Q}_{i}\}_{i=1}^{s}$ and
$\{\mtx{B}_{i}\}_{i=1}^{s}$ one at a time. We first initiate the algorithm by setting
\begin{equation}
\label{eq:def0}
\mtx{A}^{(0)} = \mtx{A}.
\end{equation}
Then step forwards, computing for $i = 1,\,2,\,\dots,\,s$ the matrices
\begin{align}
\label{eq:def1}
\mtx{Q}_{i}   =&\ \texttt{orth}\left(\mtx{A}^{(i-1)}\mtx{\Omega}_{i}\right),\\
\label{eq:def2}
\mtx{B}_{i}   =&\ \mtx{Q}_{i}^{*}\mtx{A}^{(i-1)},\\
\label{eq:def3}
\mtx{A}^{(i)} =&\ \mtx{A}^{(i-1)} - \mtx{Q}_{i}\mtx{B}_{i}.
\end{align}
We will next prove that the matrix $\mtx{\bar{Q}}_i = [\mtx{Q}_{1},\,\mtx{Q}_{2},\,\dots,\,\mtx{Q}_{i}]$ constructed
is indeed orthonormal, and that the matrix $\mtx{A}^{(i)}$ defined by (\ref{eq:def3}) is
the ``remainder'' after $i$ steps, in the sense that
$$
\mtx{A}^{(i)} =
\mtx{A} -
\bigl[\mtx{Q}_{1},\,\mtx{Q}_{2},\,\dots,\,\mtx{Q}_{i}\bigr]\,
\bigl[\mtx{Q}_{1},\,\mtx{Q}_{2},\,\dots,\,\mtx{Q}_{i}\bigr]^{*}\,
\mtx{A} = \mtx{A} - \mtx{\bar{Q}}_i \mtx{\bar{Q}}_i^* \mtx{A}.
$$
To be precise, we will prove the following proposition:

\begin{prop}
\label{prop:blocked}
Let $\mtx{A}$ be an $m\times n$ matrix.
Let $b$ denote a block size, and let $s$ denote the number of steps.
Suppose that the rank of $\mtx{A}$ is at least $sb$.
Let $\mtx{\Omega}$ be a Gaussian random matrix of size $n\times sb$, partitioned
as in (\ref{eq:OmegaBlock}), with each $\mtx{\Omega}_{j}$ of size $n\times b$.
Let $\{\mtx{A}^{(j)}\}_{j=0}^{i}$, $\{\mtx{Q}_{j}\}_{j=1}^{i}$, and $\{\mtx{B}_{j}\}_{j=1}^{i}$,
be defined by (\ref{eq:def0}) -- (\ref{eq:def3}). Set:
\begin{equation}
\label{eq:defP}
\mtx{P}_{i} = \sum_{j=1}^{i}\mtx{Q}_{j}\,\mtx{Q}_{j}^{*}.
\end{equation}
and
\begin{equation}
\mtx{\bar{Q}}_i = \bigl[\mtx{Q}_{1},\,\mtx{Q}_{2},\,\dots, \mtx{Q}_{i}\bigr] \quad \mbox{,} \quad \mtx{\bar{B}}_i = \bigl[\mtx{B}^*_{1},\,\mtx{B}^*_{2},\,\dots, \mtx{B}^*_{i}\bigr]^* \quad \mbox{,} \quad
\mtx{\bar{Y}}_i = \bigl[\mtx{A} \mtx{\Omega}_1,\,\mtx{A} \mtx{\Omega}_{2},\,\dots, \mtx{A} \mtx{\Omega}_{i}\bigr]
\end{equation}
Then for every $i = 1,2,\dots,s$, it is the case that:
\renewcommand{\labelenumi}{(\alph{enumi})}
\begin{enumerate}
\item The matrix $\mtx{\bar{Q}}_i$ is ON, so $\mtx{P}_{i}$ is an orthogonal projection.
\item $\mtx{A}^{(i)} = \bigl(\mtx{I} - \mtx{P}_{i}\bigr)\,\mtx{A} = \bigl(\mtx{I} - \mtx{\bar{Q}}_i \mtx{\bar{Q}}_i^*\bigr)\,\mtx{A}$ and $\mtx{\bar{B}}_i = \mtx{\bar{Q}}_i^* \mtx{A}$.
\item $R(\mtx{\bar{Q}}_i) =
       R(\mtx{\bar{Y}}_i)$.
\end{enumerate}
\end{prop}


\begin{proof}
The proof is by induction. We will several times use that if $\mtx{C}$ is a matrix
of size $n\times b$ of full rank, and we set $\mtx{Q} = \texttt{orth}(\mtx{C})$, then $R(\mtx{Q}) = R(\mtx{C})$,
where $R(\mtx{X})$ denotes the range of $\mtx{X}$.
We will also use the fact that if $\mtx{\Omega}$ is a Gaussian random matrix of size $n\times \ell$,
and $\mtx{E}$ is a matrix of size $m\times n$ with rank at least $\ell$, then the rank of $\mtx{E}\mtx{\Omega}$
is with probability 1 precisely $\ell$ \cite{2011_martinsson_randomsurvey}.

Direct inspection of the definitions show that (a), (b), (c) are all true for $i=1$.
Suppose all statements are true for $i-1$. We will prove that then (a), (b), (c) hold for $i$.

To prove that (a) holds for $i$, we use that (b) holds for $i-1$ and insert this into (\ref{eq:def1}) to get
\begin{equation}
\label{eq:ad}
\mtx{Q}_{i} = \texttt{orth}((\mtx{I} - \mtx{P}_{i-1})\mtx{A}\mtx{\Omega}_{i}).
\end{equation}
Then observe that $\mtx{P}_{i-1}$ is the orthogonal projection onto a space of dimension $b(i-1)$, which
means that the matrix $\mtx{A}^{(i-1)} = (\mtx{I} - \mtx{P}_{i-1})\mtx{A}$ has rank at least $bs - b(i-1) = b(s-i+1) \geq b$.
Consequently, $\mtx{A}^{(i-1)}\mtx{\Omega}_{i}$ has rank precisely $b$. This shows that
$$
R(\mtx{Q}_{i}) \subseteq R(\mtx{I} - \mtx{P}_{i-1}) = R(\mtx{P}_{i-1})^{\perp} =
R([\mtx{Q}_{1},\,\mtx{Q}_{2},\,\dots,\,\mtx{Q}_{i-1}])^{\perp}.
$$
It follows that $\mtx{Q}_{j}^{*}\mtx{Q}_{i} = \mtx{0}$ whenever $j<i$ which shows that
$\mtx{\bar{Q}}_i$ is ON. Next,
\begin{equation*}
\mtx{B}_i = \mtx{Q}_i^* \mtx{A}^{(i-1)} = \mtx{Q}_i^* \left( \mtx{I} - \mtx{Q}_{i-1} \mtx{Q}^*_{i-1} \right) \mtx{A}^{(i-2)} = \mtx{Q}_i^* \mtx{A}^{(i-2)} = \dots =  \mtx{Q}_i^* \mtx{A}^{(0)} = \mtx{Q}_i^* \mtx{A}
\end{equation*}
It follows that:
\begin{equation*}
\mtx{\bar{Q}}_i^* \mtx{A} =  [\mtx{Q}_1, \dots, \mtx{Q}_i]^* \mtx{A} =
\left[
\begin{matrix}
\mtx{Q}_1^* \mtx{A} \\
\vdots \\
\mtx{Q}_i^* \mtx{A}
\end{matrix}
\right] =
\left[
\begin{matrix}
\mtx{B}_1 \\
\vdots \\
\mtx{B}_i
\end{matrix}
\right] = \mtx{\bar{B}}_i
\end{equation*}
Thus, (a) holds for $i$.

Proving (b) is a simple calculation. Combining (\ref{eq:def2}) and (\ref{eq:def3}) we get
\begin{multline*}
\mtx{A}^{(i)} =
\mtx{A}^{(i-1)} - \mtx{Q}_{i}\mtx{Q}_{i}\mtx{A}^{(i-1)} =
(\mtx{I} - \mtx{Q}_{i}\mtx{Q}_{i}^{*})\mtx{A}^{(i-1)} \stackrel{(b)}{=}
(\mtx{I} - \mtx{Q}_{i}\mtx{Q}_{i}^{*})(\mtx{I} - \mtx{P}_{i-1})\mtx{A} =
(\mtx{I} - (\mtx{P}_{i-1} + \mtx{Q}_{i}\mtx{Q}_{i}^{*}))\mtx{A},
\end{multline*}
where in the last step we used that $\mtx{Q}_{i}^{*}\mtx{P}_{i-1} = \mtx{0}$.
Since $\mtx{P}_{i} = \mtx{P}_{i-1} + \mtx{Q}_{i}\mtx{Q}_{i}^{*}$, this proves (b).

To prove (c), we observe that (\ref{eq:ad}) implies that
\begin{equation}
\label{eq:LGG3}
R(\mtx{Q}_{i}) \subseteq
R([\mtx{A}\mtx{\Omega}_{i}, \mtx{P}_{i-1}\mtx{A}\mtx{\Omega}_{i}]).
\end{equation}
Induction assumption (c) tells us that
\begin{equation}
\label{eq:LGG4}
R(\mtx{P}_{i-1}\mtx{A}\mtx{\Omega}_{i}) \subseteq
R([\mtx{Q}_{1},\ \mtx{Q}_{2},\ \cdots,\ \mtx{Q}_{i-1}]) =
R([\mtx{A}\mtx{\Omega}_{1},\ \mtx{A}\mtx{\Omega}_{2},\ \cdots,\ \mtx{A}\mtx{\Omega}_{i-1}]).
\end{equation}
Combining (\ref{eq:LGG3}) and (\ref{eq:LGG4}), we find
\begin{equation}
\label{eq:ad2}
R(\mtx{Q}_{i}) \subseteq
R([\mtx{A}\mtx{\Omega}_{1},\ \mtx{A}\mtx{\Omega}_{2},\ \cdots,\ \mtx{A}\mtx{\Omega}_{i-1},\ \mtx{A}\mtx{\Omega}_{i}]).
\end{equation}
Equation (\ref{eq:ad2}) together with the induction assumption (c) imply that
$R([\mtx{Q}_{1},\,\mtx{Q}_{2},\,\dots,\,\mtx{Q}_{i}]) \subseteq
R([\mtx{A}\mtx{\Omega}_{1},\,\mtx{A}\mtx{\Omega}_{2},\,\cdots,\,\mtx{A}\mtx{\Omega}_{i}])$.
But both of these spaces have dimension precisely $bi$, so the fact that one is a subset
of the other implies that they must be identical.\end{proof}

Let us next compare the blocked algorithm defined by relations (\ref{eq:def0}) -- (\ref{eq:def3})
to the unblocked algorithm described in Figure \ref{fig:randQB}.
Let for a fixed Gaussian matrix $\mtx{\Omega}$, the output of the blocked version
be $\{\mtx{Q},\mtx{B}\}$ and let the output of the unblocked method be
$\{\tilde{\mtx{Q}},\,\tilde{\mtx{B}}\}$. These two pairs of matrices do not need to be identical.
(They depend on how exactly the QR factorizations are implemented, for instance).
However, Proposition \ref{prop:blocked} demonstrates that the projectors
$\mtx{Q}\mtx{Q}^{*}$ and $\tilde{\mtx{Q}}\tilde{\mtx{Q}}^{*}$ are identical.
To be precise, both of these matrices represent the orthogonal projection
onto the space $R(\mtx{A}\mtx{\Omega})$. This means that the error resulting
from the two algorithms are also identical
$$
\underbrace{\mtx{A} - \mtx{Q}\mtx{Q}^{*}\mtx{A}}_{\mbox{\textit{error from blocked algorithm}}} =
\underbrace{\mtx{A} - \tilde{\mtx{Q}}\tilde{\mtx{Q}}^{*}\mtx{A}}_{\mbox{\textit{error from non blocked algorithm}}}.
$$
Consequently, all theoretical results given in
\cite{2011_martinsson_randomsurvey}, cf.~Section \ref{sec:theory},
directly apply to the output of the blocked algorithm too.

\subsection{Adaptive rank determination}
The blocked algorithm defined by (\ref{eq:def0}) -- (\ref{eq:def3}) was in Section
\ref{sec:basicblocking} presented for the case where the rank $\ell$ of the approximation
is given in advance. A perhaps more common situation in practical applications is that
a precision $\varepsilon > 0$ is specified, and then we seek to compute an approximation
of as low rank as possible that is accurate to precision $\varepsilon$. Observe that
in the algorithm defined by (\ref{eq:def0}) -- (\ref{eq:def3}), we proved that after
step $i$ has been completed, we have
$$
\|\mtx{A}^{(i)}\| =
\|\mtx{A} - \mtx{P}_{i}\mtx{A}\| =
\|\mtx{A} - \bigl[\mtx{Q}_{1}\ \mtx{Q}_{2}\ \cdots\ \mtx{Q}_{i}\bigr]\,
            \bigl[\mtx{Q}_{1}\ \mtx{Q}_{2}\ \cdots\ \mtx{Q}_{i}\bigr]^{*}\,\mtx{A}\|.
$$
In other words, $\mtx{A}^{(i)}$ holds precisely the residual remaining after step $i$.
This means that incorporating adaptive rank determining is now trivial --- we simply
compute $\|\mtx{A}^{(i)}\|$ after completing step $i$, and break once $\|\mtx{A}^{(i)}\| \leq \varepsilon$.
The algorithm resulting is shown as \texttt{randQB\_b} in Figure \ref{fig:randQB_b}.
(The purpose of line (3') will be explained in Section \ref{sec:roundoff}.)

\begin{remark}
Recall that our default norm in this manuscript, the Frobenius norm, is simple to compute,
which means that the check on whether to break the loop on line (7) in Figure
\ref{fig:randQB_b} hardly adds at all to the execution time. If circumstances warrant the use of a
norm that is more expensive to compute, then some modification of the algorithm would be
required. Suppose, for instance, that we seek an approximation in the spectral norm. We
could then use the fact that the Frobenius norm is an upper bound on the spectral norm,
keep the Frobenius norm as the breaking condition, and then eliminate any ``superfluous''
degrees of freedom that were included in the post-processing step, cf.~Section \ref{sec:rSVD}.
(This approach would only be practicable for matrices whose singular
values exhibit reasonable decay, as otherwise the discrepancy in the $\varepsilon$-ranks
would be prohibitively large.)
\end{remark}

\subsection{Floating point arithmetic}
\label{sec:roundoff}
When the algorithm defined by (\ref{eq:def0}) -- (\ref{eq:def3}) is carried out in
finite precision arithmetic, a serious problem often arises in that round-off errors
will accumulate and will cause loss of orthonormality among the columns in
$\{\mtx{Q}_{1},\,\mtx{Q}_{2},\,\dots\}$. The problem is that as the computation proceeds,
the columns of each computed matrix $\mtx{Q}_{i}$ will due to round-off errors drift
into the span of the columns of $\{\mtx{Q}_{1},\,\mtx{Q}_{2},\,\dots,\,\mtx{Q}_{i-1}\}$.
To fix this problem, we explicitly reproject $\mtx{Q}_{i}$ away from the span of the
previously computed basis vectors \cite{BjorckReortho}. The line (3') in Figure \ref{fig:randQB_b}
represents the re-projection that is done to combat the accumulation of round-off errors.
(Note that if the algorithm is carried out in exact arithmetic, then
$\mtx{Q}_{j}^{*}\mtx{Q}_{i} = \mtx{0}$ whenever $j < i$, so line (3') would have no effect.)

\begin{figure}
\fbox{
\begin{minipage}{210mm}

\begin{tabbing}
\hspace{10mm} \= \hspace{5mm} \= \hspace{80mm} \=\kill
     \> \textbf{function} $[\mtx{Q},\mtx{B}] = \texttt{randQB\_b}(\mtx{A},\varepsilon,b)$\\[3mm]
(1)  \> \textbf{for} $i = 1,\,2,\,3,\,\dots$\\[1mm]
(2)  \> \> $\mtx{\Omega}_{i} = \texttt{randn}(n,b)$\\[1mm]
(3)  \> \> $\mtx{Q}_{i} = \texttt{orth}(\mtx{A}\mtx{\Omega}_{i})$ \> {\color{blue}$\Cmm mnb + \Cqr mb^2$}\\[1mm]
(3') \> \> $\mtx{Q}_{i} = \texttt{orth}(\mtx{Q}_{i} - \sum_{j=1}^{i-1}\mtx{Q}_{j}\mtx{Q}_{j}^{*}\mtx{Q}_{i})$
           \> {\color{blue}$2(i-1)\Cmm mb^2 + \Cqr mb^2$}\\[1mm]
(4)  \> \> $\mtx{B}_{i} = \vct{Q}_{i}^{*}\mtx{A}$
           \> {\color{blue}$\Cmm mnb$}\\[1mm]
(5)  \> \> $\mtx{A} = \mtx{A} - \mtx{Q}_{i}\mtx{B}_{i}$
           \> {\color{blue}$\Cmm mnb$}\\[1mm]
(6)  \> \> \textbf{if} $\|\mtx{A}\| < \varepsilon$ \textbf{then stop}\\[1mm]
(7)  \> \textbf{end for}\\[1mm]
(8) \> Set $\mtx{Q} = [\mtx{Q}_{1}\ \cdots\ \mtx{Q}_{i}]$ and $\mtx{B} = [\mtx{B}_{1}^{*}\ \cdots\ \mtx{B}_{i}^{*}]^{*}$.
\end{tabbing}
\end{minipage}}
\caption{A \textit{blocked} version of the randomized range finder, cf.~Fig~\ref{fig:randQB}.
The algorithm takes as input an $m\times n$ matrix $\mtx{A}$, a block size $b$, and a tolerance $\varepsilon$.
Its output are factors $\mtx{Q}$ and $\mtx{B}$ such that $\|\mtx{A} - \mtx{Q}\mtx{B}\|\leq \varepsilon$.
Note that if the algorithm is executed in exact arithmetic, then line (3') does nothing.
Text in blue refers to computational cost, see Section \ref{sec:execution1} for notation.}
\label{fig:randQB_b}
\end{figure}

\subsection{Comparison of execution times}
\label{sec:execution1}
Let us compare the computational cost of algorithms
\texttt{randQB} (Figure \ref{fig:randQB}) and
\texttt{randQB\_b} (Figure \ref{fig:randQB_b}).
To this end, let $\Cmm$ and $\Cqr$ denote the scaling
constants for the cost of executing a matrix-matrix
multiplication and a full QR factorization, respectively.
(While the algorithms only need the function
$\texttt{orth}$, cf.~Section \ref{sec:orth}, this cost is
for practical purposes the same as the cost for QR factorization.)
In other words, we assume that:
\begin{itemize}
\item Multiplying two matrices of sizes $m\times n$ and $n\times r$ costs $\Cmm\,mnr$.
\item Performing a QR factorization of a matrix of size $m\times n$, with $m\geq n$, costs $\Cqr\,mn^2$.
\end{itemize}
Note that these are rough estimates. Actual costs depend
on the actual sizes (note that the costs are dominated
by data movement rather than flops), but this model is
still instructive.
The execution time for the algorithm in Figure \ref{fig:randQB} is easily seen to be
\begin{equation}
\label{eq:T1}
T_{\texttt{randQB}} \sim 2\Cmm\,mn\ell + \Cqr\,m\ell^2.
\end{equation}
For the blocked algorithm of Figure \ref{fig:randQB_b}, we assume that
it stops after $s$ steps and set $\ell = sb$. Then
\begin{multline*}
T_{\texttt{randQB\_b}}
\sim \sum_{i=1}^{s}\left[3 \Cmm mnb + 2(i-1)\Cmm mb^2 + \Cqr 2mb^2\right] \\
\sim 3 s\Cmm mnb + s^2\Cmm mb^2 + s\Cqr 2mb^2.
\end{multline*}
Using that $sb = \ell$ we find
\begin{equation}
\label{eq:T2}
T_{\texttt{randQB\_b}} \sim 3 \Cmm mn\ell + \Cmm m\ell^2 + \frac{2}{s}\Cqr m\ell^2.
\end{equation}
Comparing (\ref{eq:T1}) and (\ref{eq:T2}), we see that the blocked algorithm
involves one additional term of $\Cmm mn\ell$, but on the other hand spends
less time executing full QR factorizations, as expected.

\begin{remark}
\label{remark:GPU}
All blocked algorithms that we present share the characteristic that they
slightly increase the amount of time spent on matrix-matrix multiplication,
while reducing the amount of time spent performing QR factorization. This
is a good trade-off on many platforms, but becomes particularly useful when
the algorithm is executed on a GPU. These massively multicore processors
are particularly efficient at performing matrix-matrix multiplications, but
struggle with communication intensive tasks such as a QR factorization.
\end{remark}

\section{A version of the method with enhanced accuracy}
\label{sec:power}

\subsection{Randomized sampling of a power of the matrix}
\label{sec:RSVDpower}
The accuracy of the basic randomized approximation scheme described in
Section \ref{sec:random}, and the blocked version of Section \ref{sec:blocked}
is well understood. The analysis of \cite{2011_martinsson_randomsurvey}
(see the synopsis in Section \ref{sec:theory}) shows that the error
$\|\mtx{A} - \mtx{A}_{\rm approx}\|$ depends strongly on the quantity
$\left(\sum_{j=k+1}^{\min(m,n)}\sigma_{j}^{2}\right)^{1/2}$. This implies
that the scheme is highly accurate for matrices whose singular values
decay rapidly, but less accurate when the ``tail'' singular values
have substantial weight. The problem becomes particularly pronounced
for large matrices. Happily, it was demonstrated in \cite{2009_szlam_power}
that this problem can easily be resolved when given a matrix with
slowly decaying singular values by simply applying a \textit{power}
of the matrix to be analyzed to the random matrix.
To be precise, suppose that we are given an $m \times n$ matrix $\mtx{A}$,
a target rank $\ell$, and a small integer $P$ (say $P=1$ or $P=2$).
Then the following formula will produce an ON matrix $\mtx{Q}$ whose
columns form an approximation to the range:
$$
\mtx{\Omega} = \texttt{randn}(n,\ell),
\qquad\mbox{and}\qquad
\mtx{Q} = \texttt{orth}\bigl(\bigl(\mtx{A}\mtx{A}^{*}\bigr)^{P}\mtx{A}\mtx{\Omega},0\bigr).
$$
The key observation here is that the matrix $\bigl(\mtx{A}\mtx{A}^{*}\bigr)^{P}\mtx{A}$
has exactly the same left singular values as $\mtx{A}$, but its singular values are
$\sigma_{j}^{2P+1}$ (observe that our objective is to build an ON-basis for the range
of $\mtx{A}$, and the optimal such basis consists of the leading left singular vectors).
Even a small value of $P$ will typically provide
enough decay that highly accurate results are attained. For a theoretical
analysis, see \cite[Sec.~10.4]{2011_martinsson_randomsurvey}.

When the ``power scheme'' idea is to be executed in floating point arithmetic,
substantial loss of accuracy happens whenever the singular values of $\mtx{A}$
have a large dynamic range. To be precise, if $\epsilon_{\rm mach}$ denotes
the machine precision, then any singular components smaller than
$\sigma_{1}\,\epsilon_{\rm mach}^{1/(2P+1)}$ will be lost. This problem
can be resolved by orthonormalizing the ``sample matrix'' between each
application of $\mtx{A}$ and $\mtx{A}^{*}$. This results in the scheme
we call \texttt{randQB\_p}, as shown in Figure \ref{fig:randQB_p}.
(Note that this scheme is virtually identical to a classical
subspace iteration with a random Gaussian matrix as the start \cite{MR1115470}.)

\begin{figure}
\fbox{
\begin{minipage}{140mm}
\begin{tabbing}
\hspace{10mm} \= \hspace{5mm} \= \hspace{65mm} \= \kill
    \> \textbf{function} $[\mtx{Q},\mtx{B}] = \texttt{randQB\_p}(\mtx{A},\ell,P)$\\[3mm]
(1) \> $\mtx{\Omega} = \texttt{randn}(n,\ell)$. \\[1mm]
(2) \> $\mtx{Q} = \texttt{orth}(\mtx{A}\mtx{\Omega})$. \>\> $\color{blue}\Cmm mn\ell + \Cqr m\ell^{2}$\\[1mm]
(3) \> \textbf{for} $j = 1:P$\\[1mm]
(4) \> \> $\mtx{Q} = \texttt{orth}(\mtx{A}^{*}\mtx{Q})$.\> $\color{blue}\Cmm mn\ell + \Cqr m\ell^{2}$\\[1mm]
(5) \> \> $\mtx{Q} = \texttt{orth}(\mtx{A}    \mtx{Q})$.\> $\color{blue}\Cmm mn\ell + \Cqr m\ell^{2}$\\[1mm]
(6) \> \textbf{end for}\\[1mm]
(7) \> $\mtx{B} = \mtx{Q}^{*}\mtx{A}$\>\> $\color{blue}\Cmm mn\ell$
\end{tabbing}

\end{minipage}}
\caption{An accuracy enhanced version of the basic randomized range finder in Figure \ref{fig:randQB}.
The algorithm takes as input an $m\times n$ matrix $\mtx{A}$, a rank $\ell$, and a ``power'' $P$ (see Section \ref{sec:RSVDpower}).
The output are matrices $\mtx{Q}$ and $\mtx{B}$ of sizes $m\times \ell$ and $\ell\times n$ such that
$\mtx{A} \approx \mtx{Q}\mtx{B}$. Higher $P$ leads to better accuracy, but also higher cost.
Setting $P=1$ or $P=2$ is often sufficient.
}
\label{fig:randQB_p}
\end{figure}

\subsection{The blocked version of the power scheme}
A blocked version of \texttt{randQB\_p} is easily obtained by a process analogous
to the one described in Section \ref{sec:basicblocking}, resulting
in the algorithm ``\texttt{randQB\_pb}'' in Figure \ref{fig:randQB_pb}.
Line (8) combats the problem of incremental loss of
orthonormality when the algorithm is executed in finite precision arithmetic,
cf.~Section \ref{sec:roundoff}.

\begin{figure}
\fbox{
\begin{minipage}{140mm}
\begin{tabbing}
\hspace{10mm} \= \hspace{5mm} \= \hspace{5mm} \= \hspace{65mm} \= \kill
    \> \textbf{function} $[\mtx{Q},\mtx{B}] = \texttt{randQB\_pb}(\mtx{A},\varepsilon,P,b)$\\[3mm]
(1) \> \textbf{for} $i = 1,\,2,\,3,\,\dots$ \\[1mm]
(2) \> \> $\mtx{\Omega}_{i} = \texttt{randn}(n,b)$. \\[1mm]
(3) \> \> $\mtx{Q}_{i} = \texttt{orth}(\mtx{A}\mtx{\Omega}_{i})$. \>\> $\color{blue}\Cmm mnb + \Cqr\,mb^{2}$\\[1mm]
(4) \> \> \textbf{for} $j = 1:P$\\[1mm]
(5) \> \> \> $\mtx{Q}_{i} = \texttt{orth}(\mtx{A}^{*}\mtx{Q}_{i})$.\> $\color{blue}\Cmm mnb + \Cqr mb^{2}$\\[1mm]
(6) \> \> \> $\mtx{Q}_{i} = \texttt{orth}(\mtx{A}    \mtx{Q}_{i})$.\> $\color{blue}\Cmm mnb + \Cqr mb^{2}$\\[1mm]
(7) \> \> \textbf{end for}\\[1mm]
(8) \> \> $\mtx{Q}_{i} = \texttt{orth}(\mtx{Q}_{i} - \sum_{j=1}^{i-1}\mtx{Q}_{j}\mtx{Q}_{j}^{*}\mtx{Q}_{i})$
          \>\> $\color{blue}2(i-1)\Cmm mb^2 + \Cqr mb^{2}$\\[1mm]
(9) \> \> $\mtx{B}_{i} = \vct{Q}_{i}^{*}\mtx{A}$
          \>\> $\color{blue}\Cmm mnb$\\[1mm]
(10) \> \> $\mtx{A} = \mtx{A} - \mtx{Q}_{i}\mtx{B}_{i}$
          \>\> $\color{blue}\Cmm mnb$\\[1mm]
(11)  \> \> \textbf{if} $\|\mtx{A}\| < \varepsilon$ \textbf{then stop}\\[1mm]
(12) \> \textbf{end while}\\[1mm]
(13) \> Set $\mtx{Q} = [\mtx{Q}_{1}\ \cdots\ \mtx{Q}_{i}]$ and $\mtx{B} = [\mtx{B}_{1}^{*}\ \cdots\ \mtx{B}_{i}^{*}]^{*}$.
\end{tabbing}
\end{minipage}}
\caption{A \textit{blocked} and \textit{adaptive} version of the accuracy enhanced algorithm shown
in Figure \ref{fig:randQB_p}. Its input and output are identical, except that we now provide a
tolerance $\varepsilon$ as an input (instead of a rank), and also a block size $b$.}
\label{fig:randQB_pb}
\end{figure}

\subsection{Computational complexity}
\label{sec:execution2}
When comparing the computational cost of
\texttt{randQB\_p} (cf.~Figure \ref{fig:randQB_p}) versus
\texttt{randQB\_pb} (cf.~Figure \ref{fig:randQB_pb}),
we use the notation that was introduced in Section \ref{sec:execution1}.
By inspection, we directly find that
$$
T_{\texttt{randQB\_p}} \sim \Cmm (2 + 2P)\,mn\ell + \Cqr(1 + 2P)m\ell^{2}.
$$
For the blocked scheme, inspection tells us that
$$
T_{\texttt{randQB\_pb}} \sim \sum_{i=1}^{s}
\left[\Cmm (3 + 2P)\,mnb + 2(i-1)\Cmm mb^2 + \Cqr(2 + 2P)mb^{2}\right].
$$
Executing the sum, and utilizing that $\ell = sb$, we get
$$
T_{\texttt{randQB\_pb}}
\sim \Cmm (3 + 2P)\,mn\ell + \Cmm m\ell^{2} + \frac{1}{s}\Cqr (2 + 2P)m\ell^{2}.
$$
In other words, the blocked algorithm again spends slightly more time executing
matrix-matrix multiplications, and quite a bit less time on qr-factorizations.
This trade is often favorable, and particularly so when the algorithm is executed
on a GPU (cf.~Remark \ref{remark:GPU}). On the other hand, when $\ell \ll n$, the benefit
to saving time on QR factorizations is minor.

\subsection{Is re-orthonormalizing truly necessary?}
\label{sec:necessary?}
Looking at algorithm \texttt{randQB\_p}, it is very
tempting to skip the intermediate QR factorizations and simply execute steps (2) -- (6) as:
\begin{center}
\fbox{\begin{minipage}{120mm}
\begin{tabbing}
\hspace{10mm} \= \hspace{5mm} \= \hspace{5mm} \= \hspace{65mm} \= \kill
(2) \> $\mtx{Y} = \mtx{A}\mtx{\Omega}$. \>\\[1mm]
(3) \> \textbf{for} $j = 1:P$\\[1mm]
(4) \> \> $\mtx{Y} = \mtx{A}\bigl(\mtx{A}^{*}\mtx{Y}\bigr)$.\\[1mm]
(5) \> \textbf{end for}\\[1mm]
(6) \> $\mtx{Q} = \texttt{orth}(\mtx{Y})$
\end{tabbing}
\end{minipage}}
\end{center}
This simplification does speed things up substantially, but as we mentioned earlier, it
can lead to loss of accuracy. In this section we state some conjectures
about when re-orthonormalization is necessary. These conjectures appear to show that
the blocked scheme is much more resilient to skipping re-orthonormalization.

To describe the issue, let us fix a (small) integer $P$, and define the matrix
$$
\mtx{A}_{P} = \bigl(\mtx{A}\mtx{A}^{*}\bigr)^{P}\mtx{A}.
$$
If the SVD of $\mtx{A}$ is $\mtx{A} = \mtx{U}\mtx{\Sigma}\mtx{V}^{*}$, then the SVD of $\mtx{A}_{P}$ is
$$
\mtx{A}_{P} = \mtx{U}\mtx{\Sigma}^{2P+1}\mtx{V}^{*}.
$$
In computing $\mtx{A} = \mtx{A}_{P}\mtx{\Omega}$,
we lose all information about any singular mode $i$ for which $\sigma_{i}^{2P+1} \leq \sigma_{1}^{2P+1}\epsilon_{\rm mach}$,
where $\epsilon_{\rm mach}$ is the machine precision. In other words, in order to accurately
resolve the first $k$ singular modes, re-orthogonalization is needed if
\begin{equation}
\label{eq:open1}
\frac{\sigma_{1}}{\sigma_{k}} > \epsilon_{\rm mach}^{1/(2P+1)}.
\end{equation}
As an example, with $P=2$ and $\epsilon_{\rm mach} = 10^{-15}$, we find that $\epsilon_{\rm mach}^{1/(2P+1)} = 10^{-3}$,
so re-orthonormalization is imperative resolve any components smaller than $\sigma_{1} \cdots 10^{-3}$. Moreover,
if we skip re-orthonormalization, we are likely to see an overall loss of accuracy affecting singular values and
singular vectors associated with larger singular values.

Next consider the blocked scheme. The crucial observation is that now, instead of trying to extract
the whole range of singular values $\{\sigma_{j}\}_{j=1}^{k}$ (and their associated eigenvectors) at
once, we now extract them in $s$ groups of $b$ modes each, where $k \approx sb$. This means that
we can expect to get reasonable accuracy as long as
\begin{equation}
\label{eq:open2}
\frac{\sigma_{(i-1)b+1}}{\sigma_{ib}} \leq \epsilon_{\rm mach}^{1/(2P+1)},
\qquad\mbox{for}\ i = 1,\,2,\,\dots,\,s.
\end{equation}
Comparing (\ref{eq:open1}) and (\ref{eq:open2}), we see that (\ref{eq:open2}) is a much milder condition,
in particular when the block size $b$ is much smaller than $k$.

All claims in this section are \textit{heuristics.} However, while they have not been rigorously proven,
they are supported by extensive numerical experiments, see Section \ref{sec:skipqr}.

\section{Numerical experiments}

In this section, we present numerical examples that illustrate the computational efficiency
and the accuracy of the proposed scheme, see Sections \ref{sec:speed} and \ref{sec:accuracy},
respectively. The codes we used are available at \texttt{http://amath.colorado.edu/faculty/martinss/main\_codes.html}
and we encourage any interested reader to try the methods out, and explore different parameter
sets than those included here.

\subsection{Comparison of execution speeds}
\label{sec:speed}

We first compare the run times of different techniques for computing a partial (rank $k$) QR factorization
of a given matrix $\mtx{A}$ of size $n\times n$. Observe that the choice of matrix is immaterial for a run
time comparison (we investigate accuracy in Section \ref{sec:accuracy}). We compared three  sets of
techniques:
\begin{itemize}
\item Truncating a \textit{full} QR factorization, computed using the Intel MKL libraries.
\item Taking $k$ steps of a column pivoted Gram-Schmidt process.
The implementation was accelerated by using MKL library functions whenever practicable.
\item The blocked ``QB'' scheme, followed by postprocessing of the factors to obtain a
QR factorization. We used the ``power method'' described in Section \ref{sec:power}
with parameters $P=0,1,2$.
\end{itemize}
The algorithms were all implemented in C and run on a desktop with a 6-core Intel Xeon E5-1660 CPU (3.30 GHZ), and 128GB of RAM.
We also ran the blocked ``QB'' scheme on an NVIDIA Tesla K40c GPU installed on the same machine,
using the Matlab GPU computing interface.
The results are shown in Figure \ref{fig:alg_timings}. Figure \ref{fig:alg_timings2} shows the dependence of the
runtime on the block size.

Figure \ref{fig:alg_timings} shows that our blocked algorithms (blue, magenta, and cyan lines)
compare favorably to both of the two benchmarks we chose --- full QR using MKL libraries (green)
and partial factorization using column pivoting (green). However, it must be noted that our
implementation of column pivoted QR is \textit{far} slower than the built-in QR factorization
in the MKL libraries. Even for as low of a rank as $k=100$, we do not break even with a full
factorization until $n=8\,000$. This implies that column pivoting can be implemented far more
efficiently than we were able to. The point is that in order to attain the efficiency of the
MKL libraries, very careful coding that is customized to any particular computing platform would
have to be done. In contrast, our blocked code is able to exploit the very high efficiency of
the MKL libraries with minimal effort.

Finally, it is worth nothing how particularly efficient our blocked algorithms are when executed
on a GPU. We gain a substantial integer factor speed-up over CPU speed in every test we conducted.

\begin{figure}
\centerline{
\includegraphics[height=80mm]{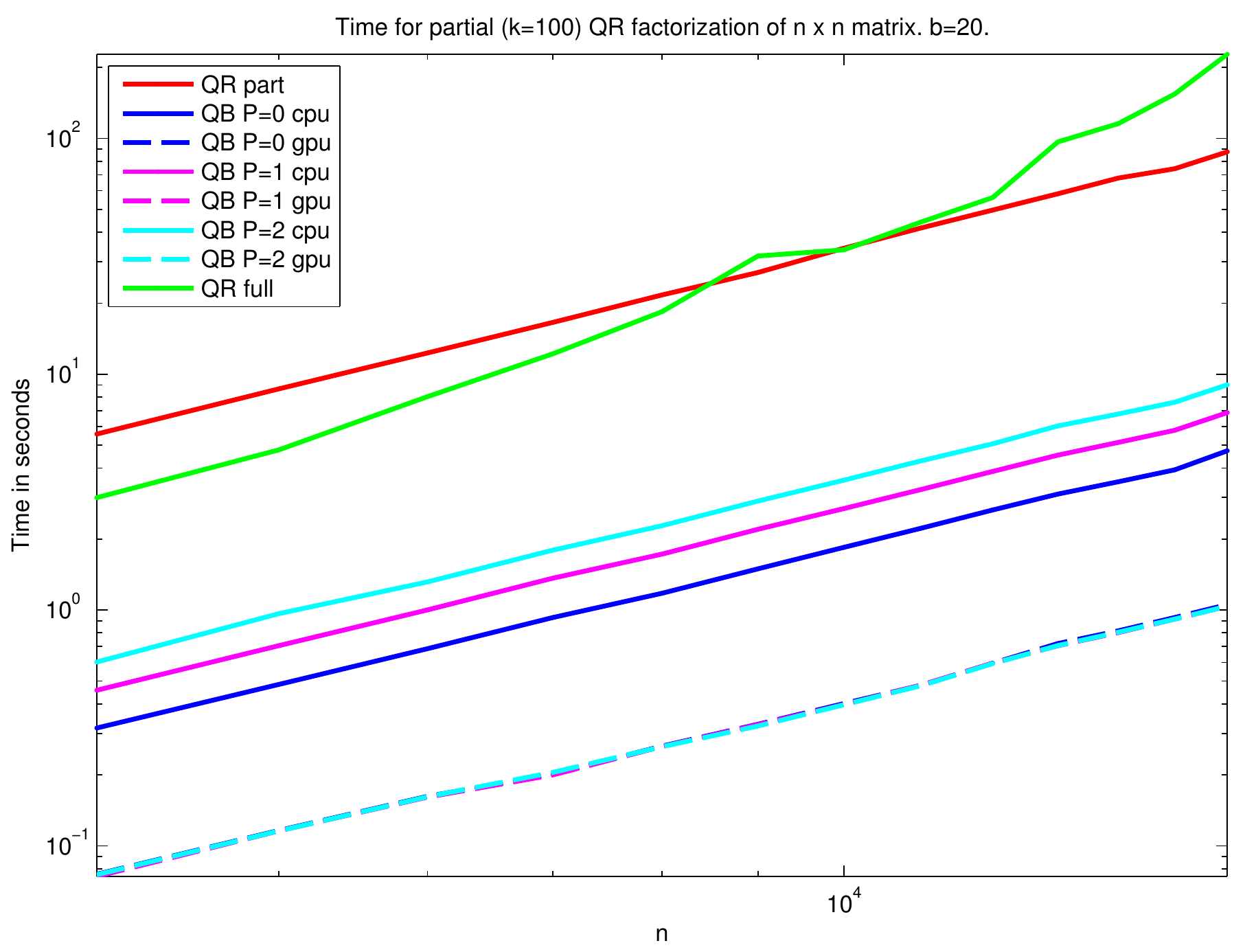}
}
\caption{Timing results for different algorithms on CPU and GPU.
The integer $P$ denotes the parameter in the ``power scheme'' described in Section \ref{sec:power}.}
\label{fig:alg_timings}
\end{figure}

\begin{figure}
\centerline{
\includegraphics[height=80mm]{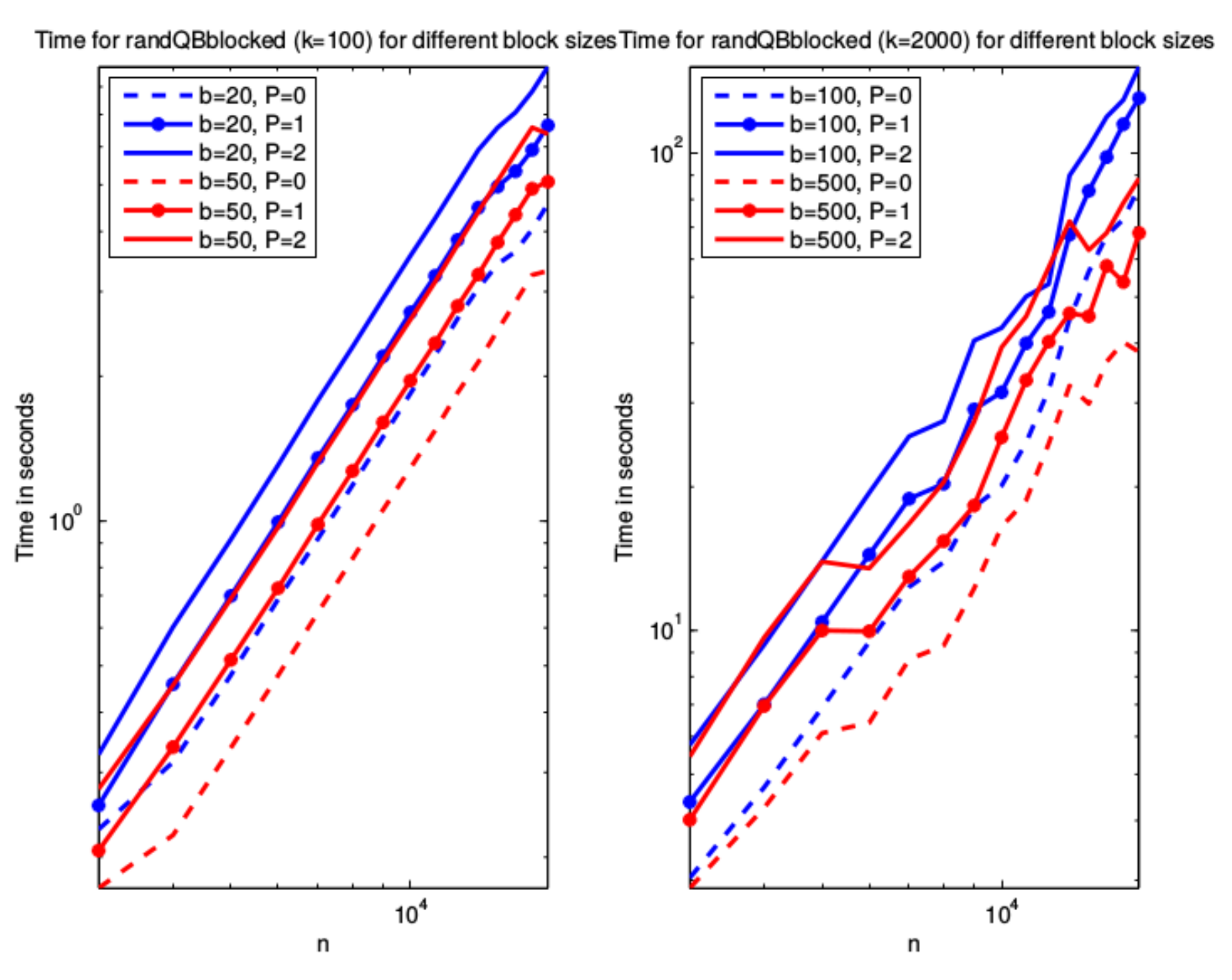}
}
\caption{Timing results for blocked QB scheme for different block sizes $b$.
The integer $P$ denotes the parameter in the ``power scheme'' described in Section \ref{sec:power}.}
\label{fig:alg_timings2}
\end{figure}

\subsection{Accuracy of the randomized scheme}
\label{sec:accuracy}
We next investigate the accuracy of the randomized schemes versus column-pivoted QR
on the one hand (easy to compute, not optimal) and versus the truncated SVD on the other
(expensive to compute, optimal). We used 5 classes of test matrices that each have
different characteristics:
\begin{description}
\item[Matrix 1 (fast decay)] Let $\mtx{A}_{1}$ denote an $m\times n$ matrix of the form
$\mtx{A}_{1} = \mtx{U}\mtx{D}\mtx{V}^{*}$ where $\mtx{U}$ and $\mtx{V}$ are randomly
drawn matrices with orthonormal columns (obtained by performing qr on a random
Gaussian matrix), and where $\mtx{D}$ is a diagonal matrix with entries roughly
given by $d_{j} = g_{j}^{2}\,\beta^{j-1}$ where $g_{j}$ is a random number drawn
from a uniform distribution on $[0,1]$ and $\beta = 0.65$. To precision $10^{-15}$,
the rank of $\mtx{A}_{1}$ is about 75.
\item[Matrix 2 (slow decay)] The matrix $\mtx{A}_{2}$ is formed just like $\mtx{A}_{1}$,
but now the diagonal entries of $\mtx{D}$ decay very slowly, with $d_{j} = (1 + 200(j-1))^{1/2}$.
\item[Matrix 3 (sparse)] The matrix $\mtx{A}_{3}$ is a sparse matrix given by
$\mtx{A}_{3} = \sum_{j=1}^{10} \frac{2}{j}\,\vct{x}_{j}\,\vct{y}_{j}^{*} + \sum_{j=11}^{\min(m,n)}\frac{1}{j}\,\vct{x}_{j}\,\vct{y}_{j}^{*}$
where $\vct{x}_{j}$ and $\vct{y}_{j}$ are random sparse vectors generated by the Matlab commands
$\texttt{sprand}(m,1,0.01)$ and $\texttt{sprand}(n,1,0.01)$, respectively. We used
$m=800$ and $n=600$ with resulted in a matrix with roughly $6\%$ non-zero elements.
This matrix was borrowed from Sorensen and Embree \cite{sorensen2014deim} and is an
example of a matrix for which column pivoted Gram-Schmidt performs particularly well.
\item[Matrix 4 (Kahan)] This is a variation of the ``Kahan counter-example'' which
is a matrix designed so that Gram-Schmidt performs particularly poorly.
The matrix here is formed via the matrix matrix product $\mtx{S}\mtx{K}$ where:
\begin{equation*}
\mtx{S} = \left[\begin{matrix}
1 & 0 & 0 & 0 & \cdots \\
0 & \zeta & 0 & 0 & \cdots \\
0 & 0 & \zeta^2 & 0 & \cdots \\
0 & 0 & 0 & \zeta^3 & \cdots \\
\vdots & \vdots & \vdots & \vdots & \ddots \\
\end{matrix}\right] \quad \mbox{and} \quad
\mtx{K} = \begin{bmatrix}
1 & -\phi & -\phi & -\phi & \cdots \\
0 & 1 & -\phi & -\phi & \cdots \\
0 & 0 & 1 & -\phi & \cdots \\
0 & 0 & 0 & 1 & \cdots \\
\vdots & \vdots & \vdots & \vdots & \ddots \\
\end{bmatrix}
\end{equation*}
with random $\zeta,\phi>0,\ \zeta^2+\phi^2=1$.
Then $\mtx{S} \mtx{K}$ is upper triangular, and for many choices of $\zeta$ and $\phi$,
classical column pivoting will yield poor performance as the different column norms will be similar and pivoting will
generally fail. The rank-$k$ approximation resulting from column pivoted QR is
substantially less accurate than the optimal rank-$k$ approximation resulting from
truncating the full SVD \cite{gu1996}. However, we obtain much better results than QR with the QB algorithm.

\item[Matrix 5 (S shaped decay)] The matrix $\mtx{A}_{5}$ is built in the
same manner as $\mtx{A}_{1}$ and $\mtx{A}_{2}$, but now the diagonal entries
of $\mtx{D}$ are chosen to first hover around 1, then decay rapidly, and then
level out at a relatively high plateau, cf.~Figure \ref{fig:acc_matrix5}.
\end{description}

We compare four different techniques for computing a rank-$k$ approximation to
our test matrices:
\begin{description}
\item[SVD] We computed the full SVD (using the Matlab command \texttt{svd}) and
then truncated to the first $k$ components.
\item[Column-pivoted QR] We implemented this using modified Gram-Schmidt with
reorthogonalization to ensure that orthonormality is strictly maintained in the
columns of $\mtx{Q}$.
\item[\texttt{randQB} --- single vector] This is the greedy algorithm labeled
``Algorithm 1'' in Section \ref{sec:greedy}, implemented with $\vct{q}_{j}$
on line (4) chosen as $\vct{q}_{j} = \vct{y}/\|\vct{y}\|$ where
$\vct{y} = \bigl(\mtx{A}\mtx{A}^{*}\bigr)^{P}\,\mtx{A}\,\vct{\omega}$ and
where $\vct{\omega}$ is a random Gaussian vector.
\item[\texttt{randQB} --- blocked] This is the algorithm \texttt{randQB\_pb}
shown in Figure \ref{fig:randQB_pb}.
\end{description}
The results are shown in Figure \ref{fig:acc_matrix1} -- \ref{fig:acc_matrix5}.
We make three observations: (1) When the ``power method'' described in Section
\ref{sec:power} is used, the accuracy of \texttt{randQB\_pb} exceeds that of
column-pivoted QR in every example we tried, even for as low of a power as $P=1$.
(2) Blocking appears to lead to no loss of accuracy. In most cases, there is
no discernible difference in accuracy between the blocked and the non-blocked
versions. (3) The accuracy of \texttt{randQB\_pb} is particularly good
when errors are measured in the Frobenius norm. In almost all cases we
investigated, essentially optimal results are obtained even for $P=1$.

Figures \ref{fig:acc_matrix1} -- \ref{fig:acc_matrix5} report the errors resulting
from a single instantiation of the randomized algorithm. Appendix \ref{app:stats}
provides more details on the statistical distribution of errors.

\begin{figure}
\includegraphics[width=\textwidth]{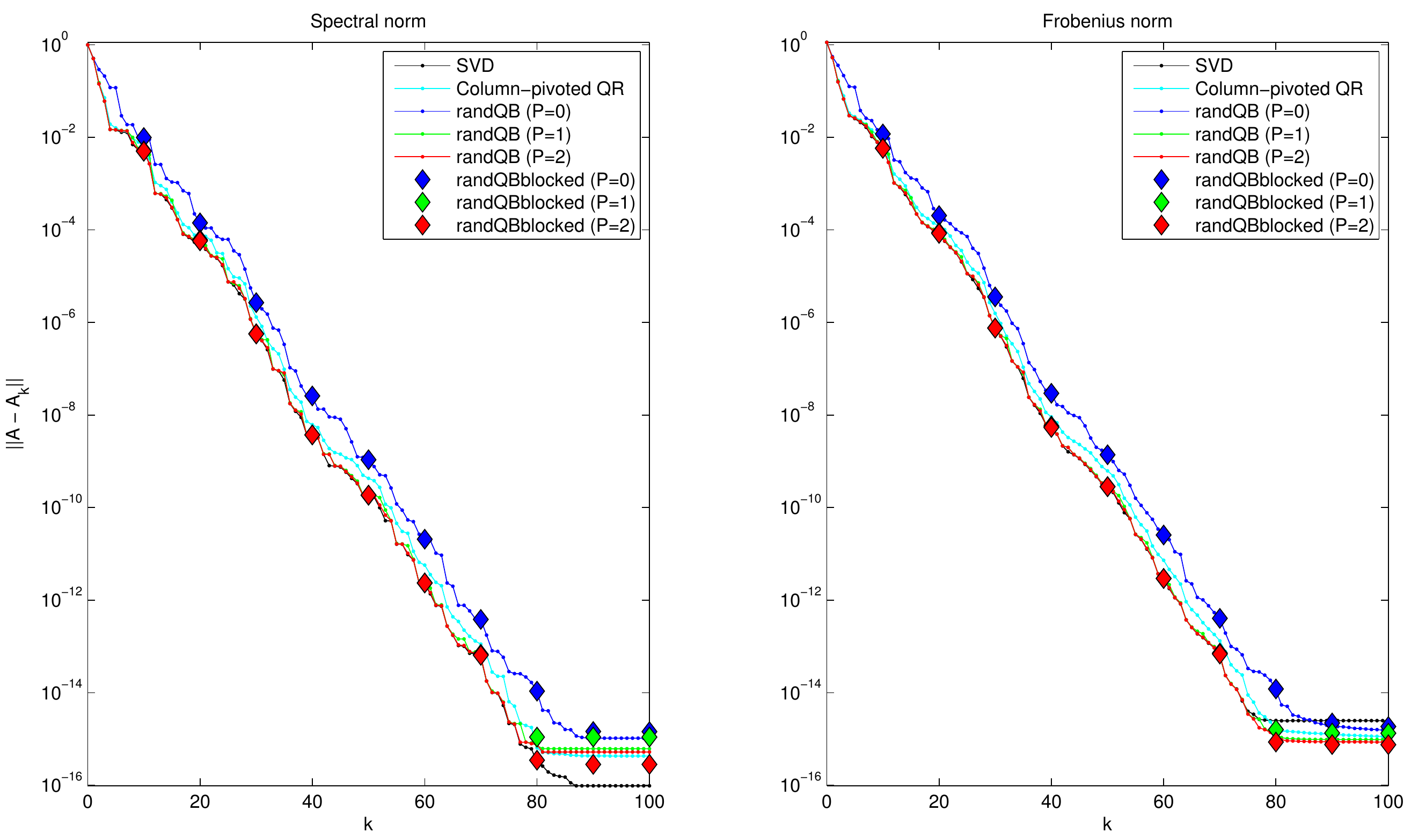}
\caption{Errors for the $800\times 600$ ``Matrix 1'' whose singular values decay very rapidly.
The block size is $b = 10$.}
\label{fig:acc_matrix1}
\end{figure}

\begin{figure}
\includegraphics[width=\textwidth]{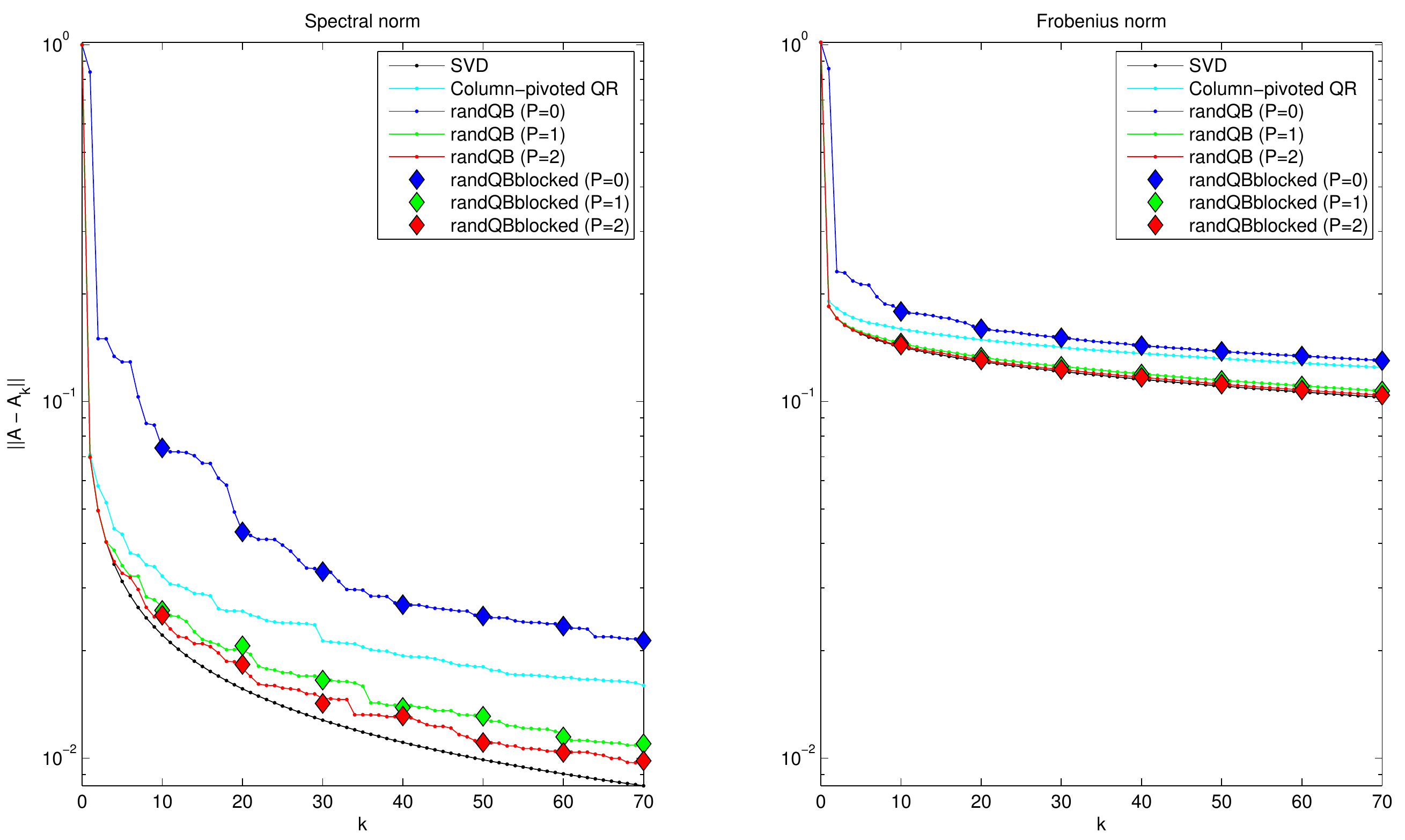}
\caption{Errors for the $800\times 600$ ``Matrix 2'' whose singular values decay slowly.
The block size is $b = 10$.}
\label{fig:acc_matrix2}
\end{figure}

\begin{figure}
\includegraphics[width=\textwidth]{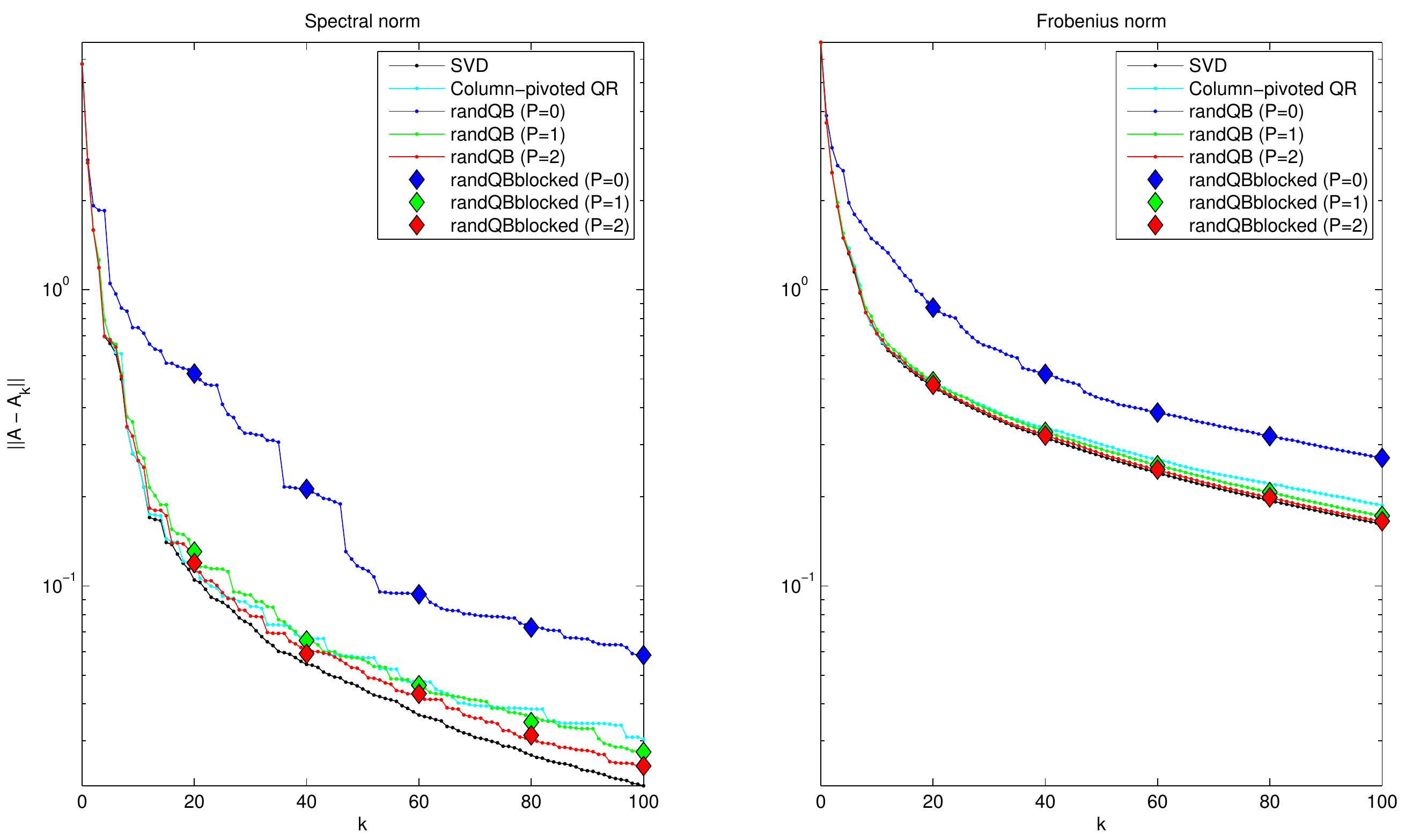}
\caption{Errors for the $800\times 600$ ``Matrix 3.'' This is a sparse matrix for
which column-pivoted Gram-Schmidt performs exceptionally well. However,
\texttt{randQB} still gives better accuracy whenever a power $P \geq 1$ is used.}
\label{fig:acc_matrix3}
\end{figure}

\begin{figure}
\includegraphics[width=\textwidth]{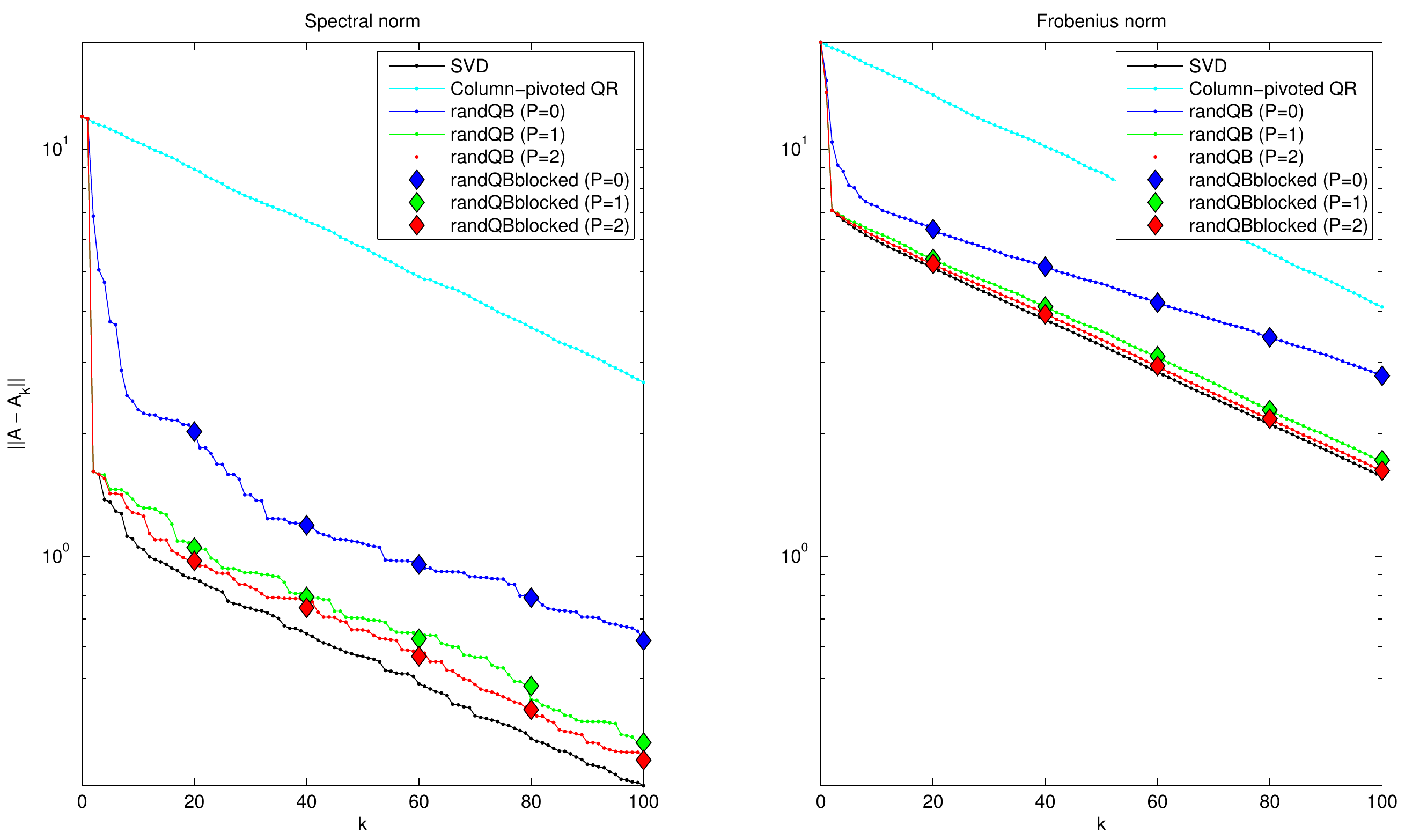}
\caption{Errors for the $1\,000\times 1\,000$ ``Matrix 4.'' This matrix is a variation of the
``Kahan counter-example'' and is designed specifically to give poor performance for column
pivoted QR. Here $b = 20$.}
\label{fig:acc_matrix4}
\end{figure}

\begin{figure}
\includegraphics[width=\textwidth]{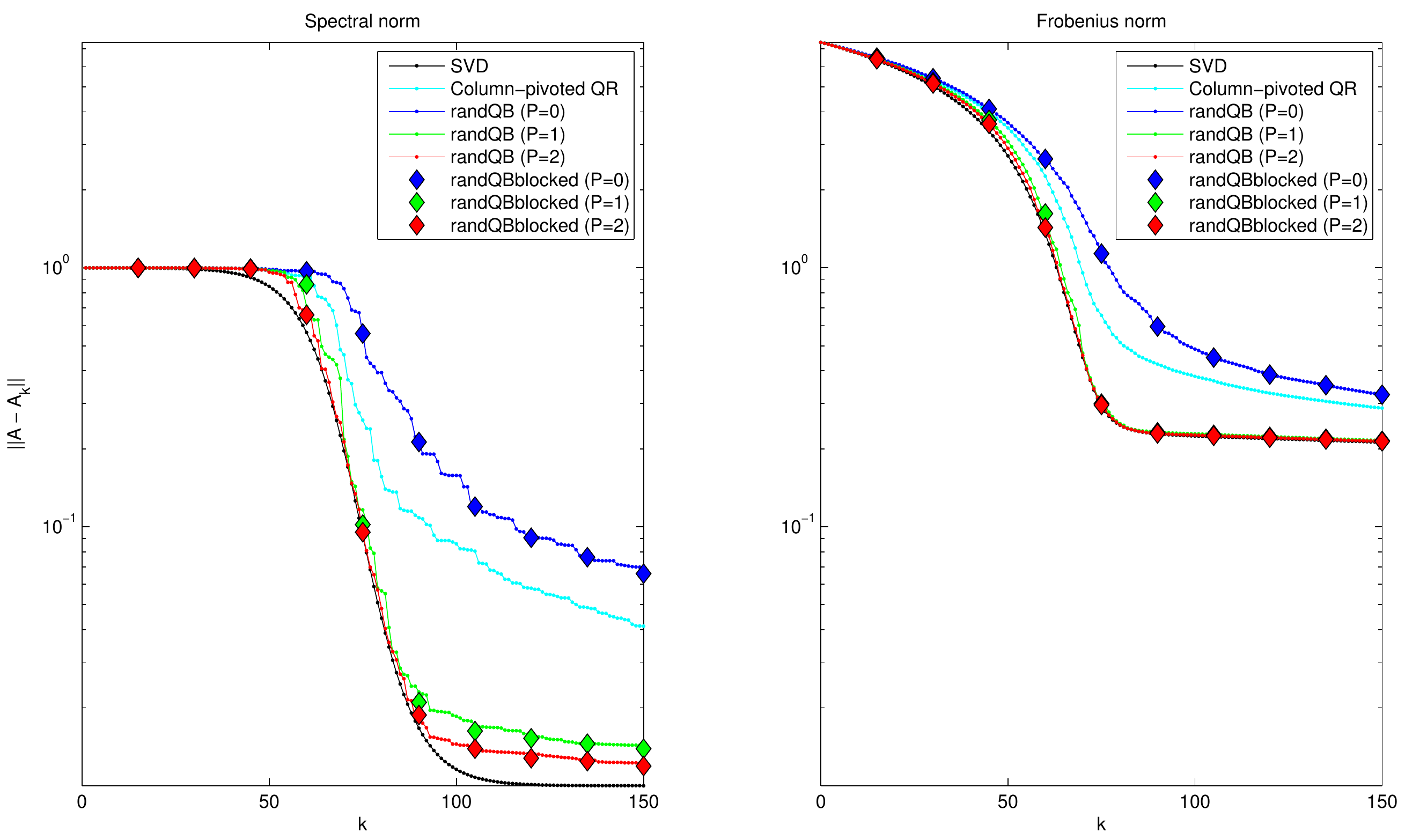}
\caption{Errors for the $800\times 600$ ``Matrix 5'' whose singular values show an ``S-shaped'' decay.
Here $b = 15$.}
\label{fig:acc_matrix5}
\end{figure}

\newpage
\subsection{When re-orthonormalization is required}
\label{sec:skipqr}
We claimed in Section \ref{sec:necessary?} that the blocked scheme is more
robust towards loss of orthonormality than the non-blocked scheme presented
in \cite{2011_martinsson_randomsurvey}. To test this hypothesis, we tested
what happens if we skip the re-orthonormalization between applications of
$\mtx{A}$ and $\mtx{A}^{*}$ in the algorithms shown in Figures
\ref{fig:randQB_p} and \ref{fig:randQB_pb}. The results are shown
in Figure \ref{fig:skipqr}. The key observation here is that the blocked versions
of \texttt{randQB} \textit{still} always yield excellent precision. When the block
size is large, the convergence is slowed down a bit compared to the more
meticulous implementation, but essentially optimal accuracy is nevertheless obtained
relatively quickly.

\begin{remark}
The numerical results in Figure \ref{fig:skipqr} substantiate the claim that
for the unblocked version, the best accuracy attainable is
$\sigma_{1}\,\epsilon_{\rm mach}^{1/(2P + 1)}$. In all examples, we have
$\sigma_{1} = 1$, so the prediction is that for $P=1$ the maximum precision is
$\bigl(10^{-15}\bigr)^{1/3} = 10^{-5}$ and for $P=2$ it is
$\bigl(10^{-15}\bigr)^{1/5} = 10^{-3}$. The results shown precisely follow this
pattern. Observe that for $\mtx{A}_{2}$, no loss of accuracy is seen at all since
the singular values we are interested in level out at about $10^{-2}$.
\end{remark}

\begin{figure}
\begin{tabular}{ccc}
\includegraphics[width=0.45\textwidth]{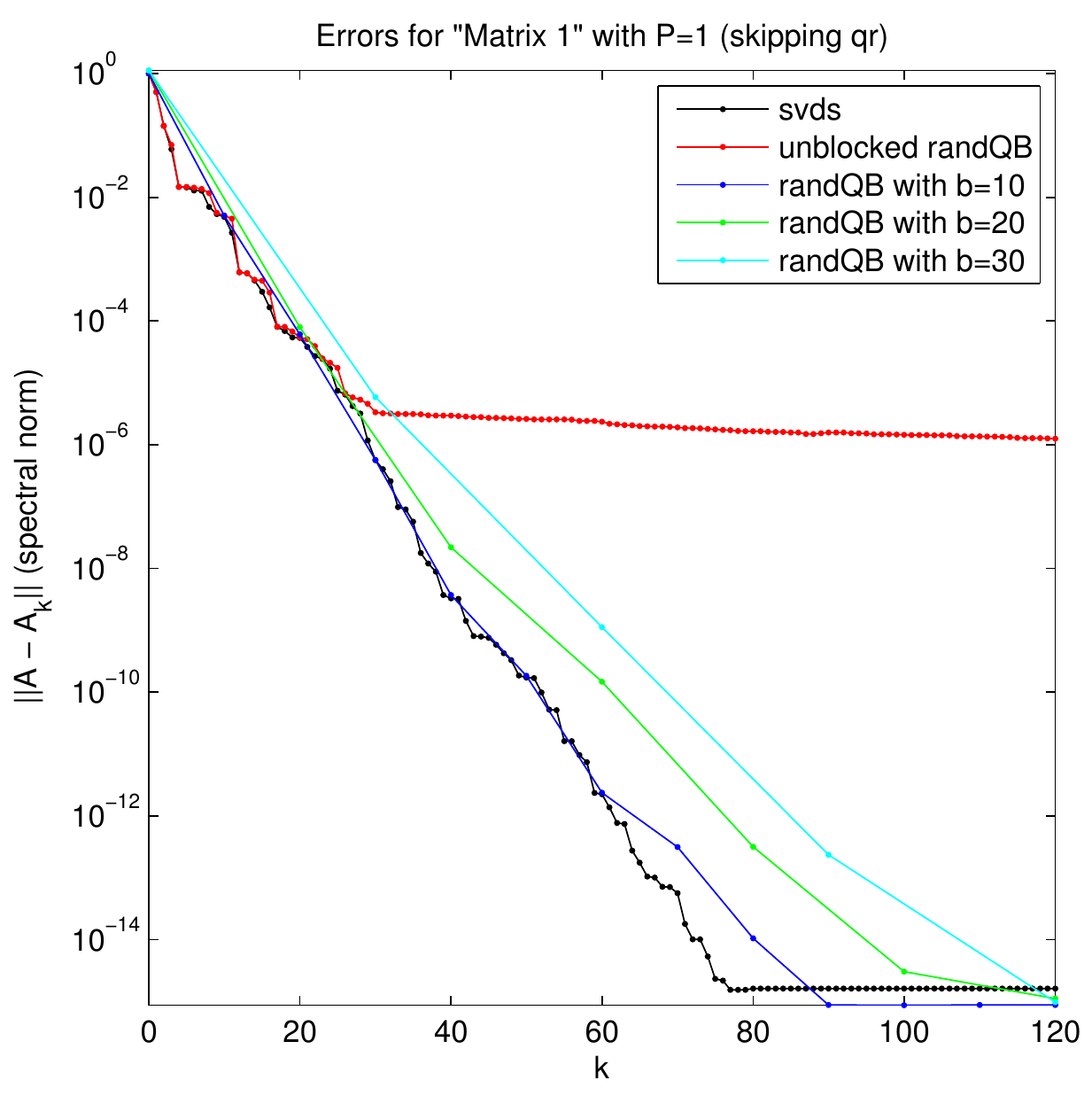}
&&
\includegraphics[width=0.45\textwidth]{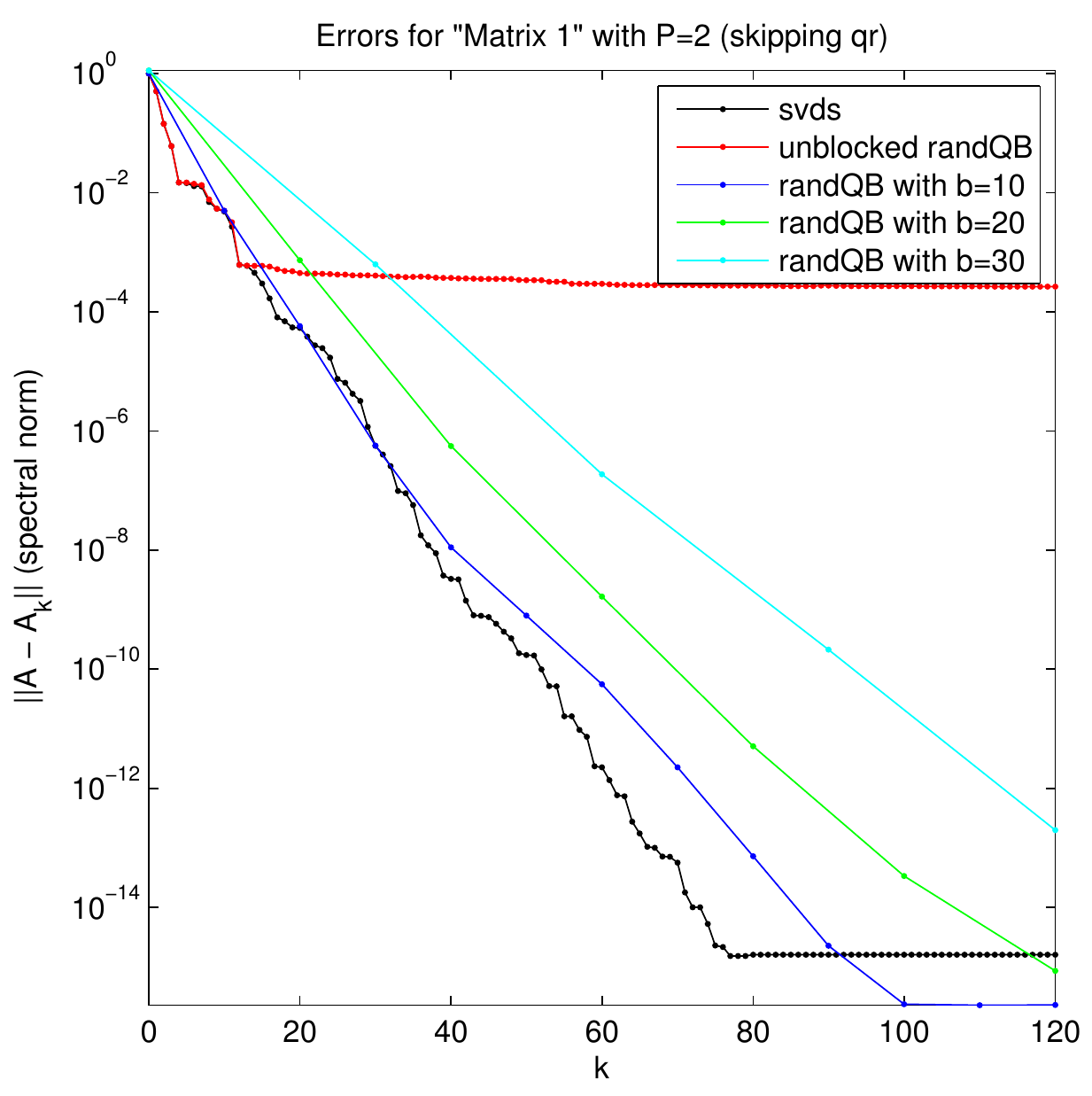}\\
\includegraphics[width=0.45\textwidth]{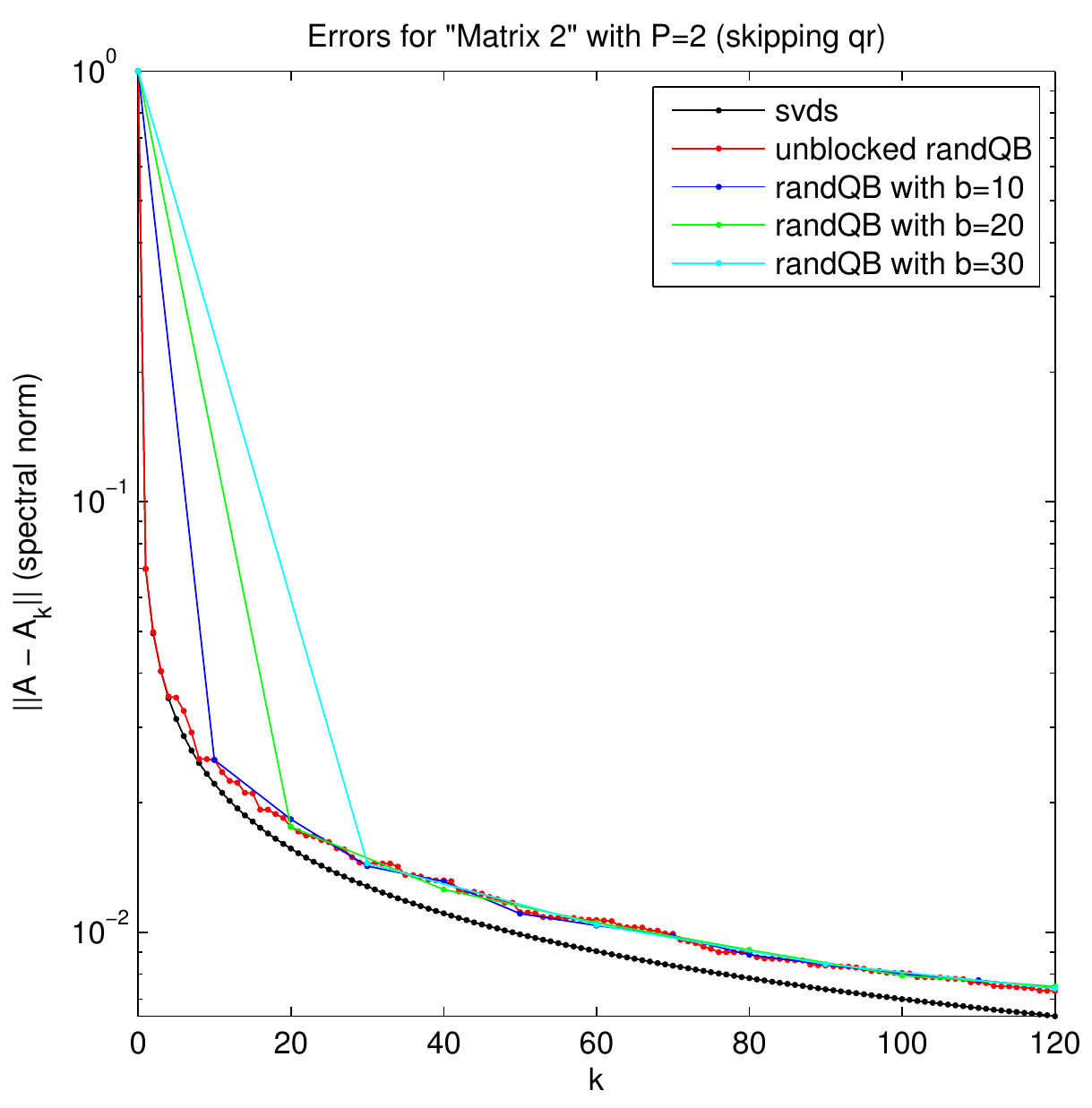}
&&
\includegraphics[width=0.45\textwidth]{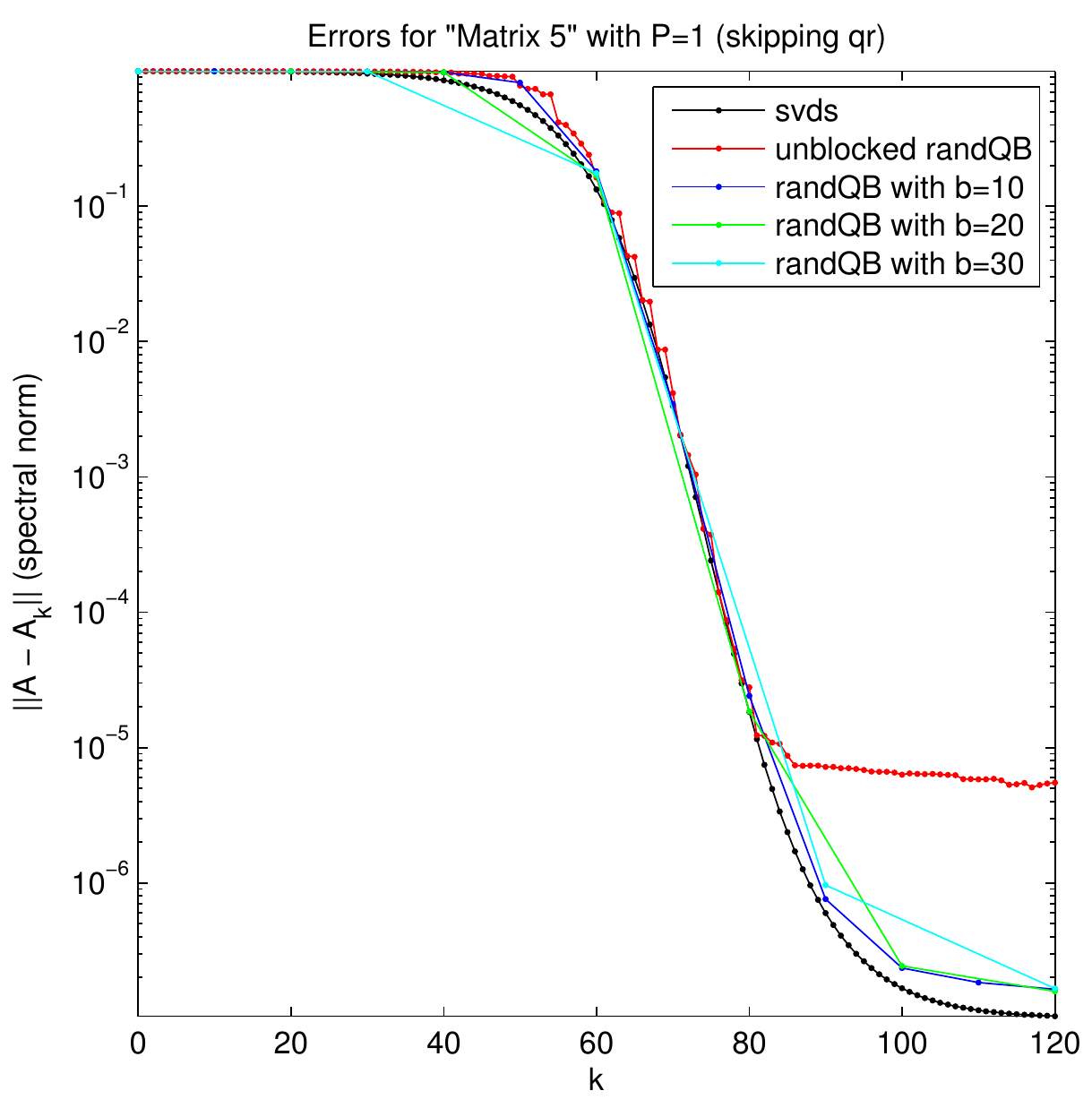}
\end{tabular}
\caption{Errors incurred when not re-orthonormalizing between applications of
$\mtx{A}$ and $\mtx{A}^{*}$ in the ``power method,'' cf.~Sections
\ref{sec:necessary?} and \ref{sec:skipqr}. 
The non-blocked scheme (red) performs precisely as predicted, and cannot resolve
anything beyond precision $10^{-5}$ when $P=1$ and $10^{-3}$ when $P=2$. The blocked
version converges slightly slower when skipping re-orthonormalization but always
reaches full precision.}
\label{fig:skipqr}
\end{figure}

\newpage
\section{Concluding remarks}

We have described a randomized algorithm for the low rank approximation of
matrices. The algorithm is based on the randomized sampling paradigm described in
\cite{2006_martinsson_random1_orig,2009_szlam_power,2011_martinsson_randomsurvey,2011_martinsson_random1}.
In this article, we introduce a \textit{blocking} technique which allows
us to incorporate adaptive rank determination without sacrificing computational
efficiency, and an \textit{updating} technique that allows us to replace the
randomized stopping criterion proposed in \cite{2011_martinsson_randomsurvey}
with a deterministic one. Through theoretical analysis and numerical examples,
we demonstrate that while the blocked scheme is mathematically equivalent to
the non-blocked scheme of \cite{2006_martinsson_random1_orig,2009_szlam_power,2011_martinsson_randomsurvey,2011_martinsson_random1}
when executed in exact arithmetic, the blocked scheme is slightly more robust towards
accumulation of round-off errors.

The updating strategy that we propose is directly inspired by a classical
scheme for computing a partial QR factorization via the column pivoted
Gram-Schmidt process. We demonstrate that the randomized version that we
propose is more computationally efficient than this classical scheme (since
it is hard to block the column pivoting scheme). Our numerical experiments
indicate that the randomized version not only improves speed, but also leads
to higher accuracy. In fact, in all examples we present, the errors resulting
from the blocked randomized scheme are very close to the optimal error obtained
by truncating a full singular value decomposition. In particular, when errors
are measured in the Frobenius norm, there is almost no loss of accuracy at all
compared to the optimal factorization, even for matrices whose singular values
decay slowly.

The scheme described can output any of the standard low-rank factorizations of \
matrices such as, e.g., a partial QR or SVD
factorization. It can also with ease produce less standard factorizations such
as the ``CUR'' and ``interpolative decompositions (ID)'',
cf.~Section \ref{sec:standard_factorizations}.


\textit{\textbf{Acknowledgements:}} The research reported was supported by DARPA,
under the contract N66001-13-1-4050, and by the NSF, under the contract DMS-1407340.

\bibliography{main_bib}
\bibliographystyle{amsplain}

\begin{appendix}

\section{Distribution of errors}
\label{app:stats}

The output of our randomized blocked approximation algorithms is a random variable,
since it depends on the drawing of a Gaussian matrix $\mtx{\Omega}$. It has been
proven (see, e.g., \cite{2011_martinsson_randomsurvey}) that due to concentration of
mass, the variation in this random variable is tiny. The output is for practical
purposes always very close to the expectation of the output. For this reason, when
we compared the accuracy of the randomized method to classical methods in
Section \ref{sec:accuracy}, we simply presented the results from one particular
draw of $\mtx{\Omega}$. In this section, we provide some more detailed numerical
experiments that illuminate exactly how little variation there is in the output of
the algorithm. We present the results from all matrices considered in Section \ref{sec:accuracy}
except Matrix 4 (the so called ``Kahan counter example'') since this is an artificial
example concocted specifically to give poor results for column-pivoted Gram Schmidt.

Figures \ref{fig:mean_err_matrix1} through \ref{fig:many_err_matrix5} provide more
information about the statistics of the outcome for the experiments reported for a
single instantiation in Figures \ref{fig:acc_matrix1} through \ref{fig:acc_matrix5}.
For each experiment, we show both the empirical expectation of the accuracy, and
the error paths from $25$ different instantiations. We observe that the errors are in
all cases tightly clustered, in particular for $P=1$ and $P=2$. We also observe that
when the singular values decay slowly, the clustering is stronger in the Frobenius norm
than in the spectral norm.

In our final set of experiments, we increased the number of experiments from $25$
to $2\,000$. To keep the plots legible, we plot the errors only for a fixed
value of $k$, see Figure \ref{fig:scatter_plots}. These experiments further
substantiate our claim that the results are tightly clustered, in particular
when the ``power method'' is used, and when the Frobenius norm is used.

\clearpage

\begin{figure}
\includegraphics[width=\textwidth]{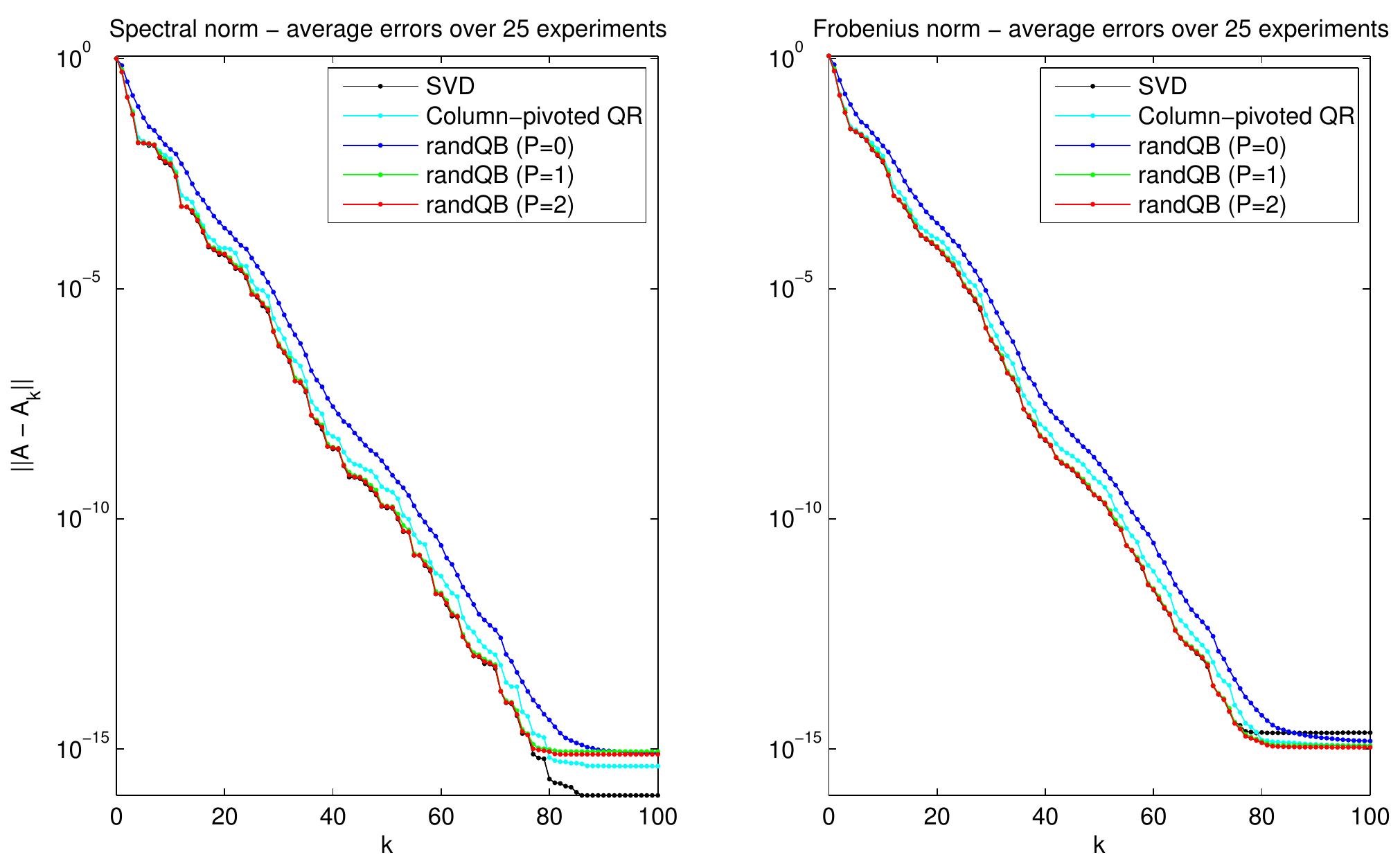}
\caption{The empirical mean errors from 25 instantiations of the randomized
factorization algorithm applied to Matrix 1.}
\label{fig:mean_err_matrix1}
\end{figure}

\begin{figure}
\includegraphics[width=\textwidth]{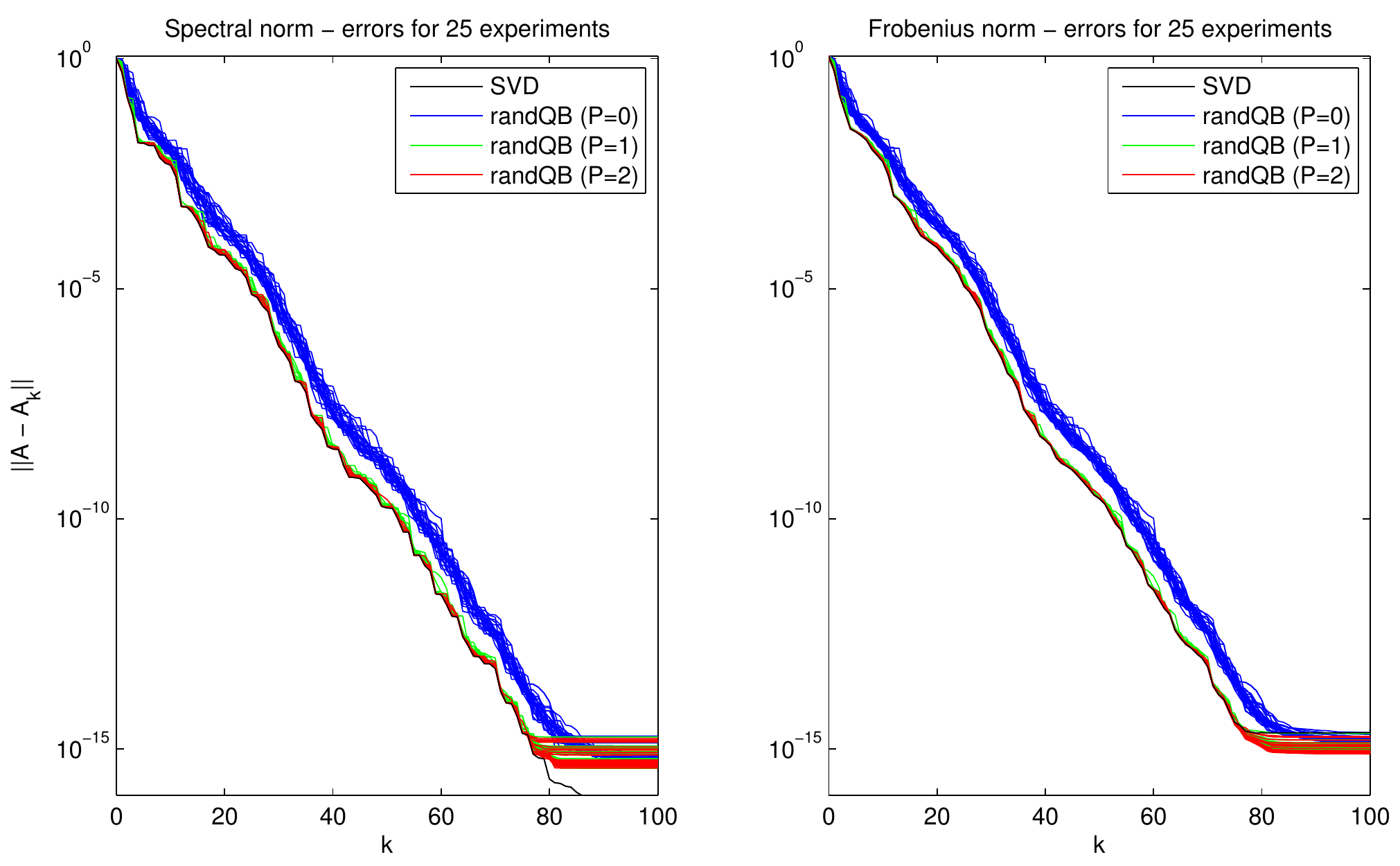}
\caption{The error paths for 25 instantiations of the randomized
factorization algorithm applied to Matrix 1.}
\label{fig:many_err_matrix1}
\end{figure}

\begin{figure}
\includegraphics[width=\textwidth]{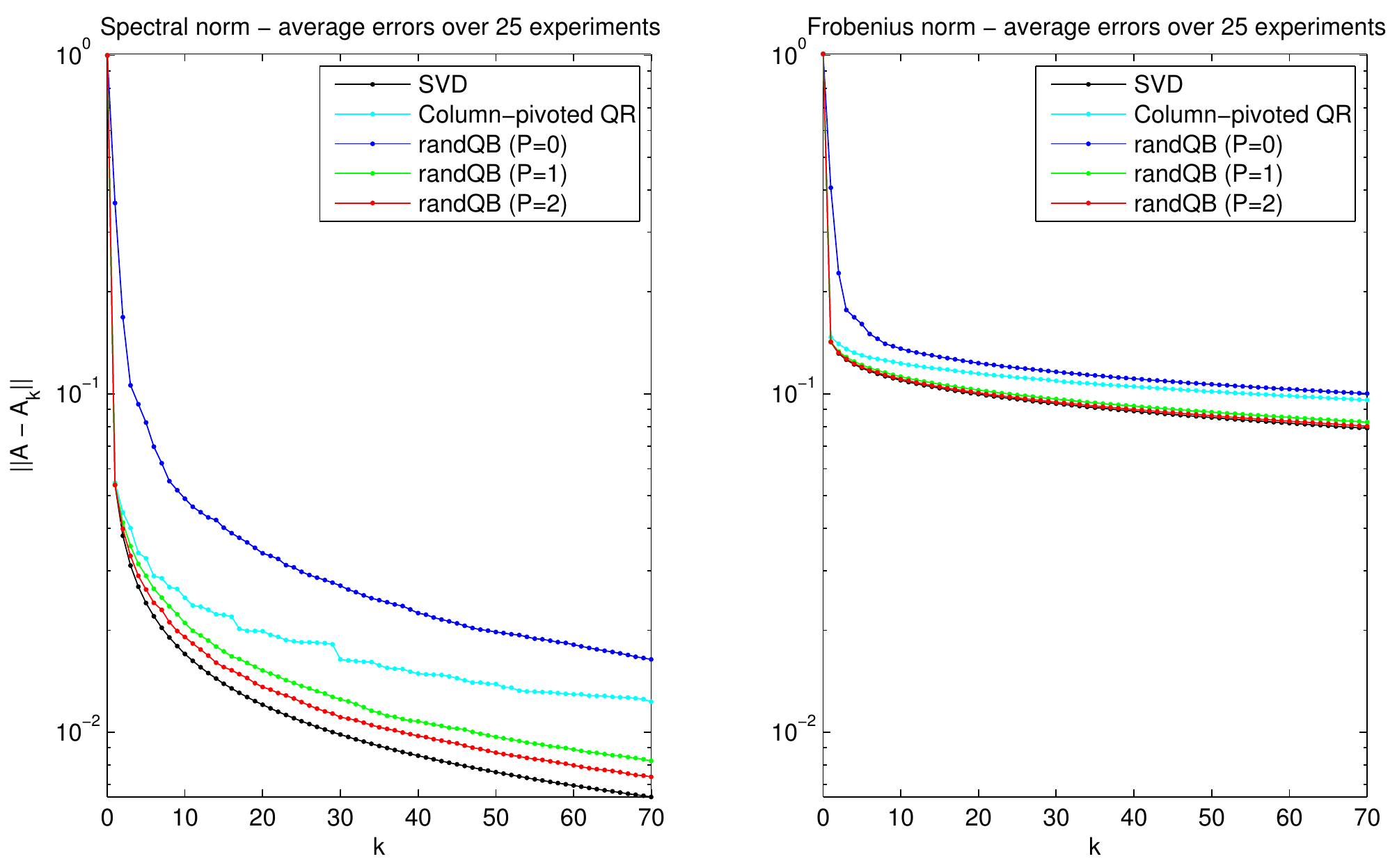}
\caption{The empirical mean errors from 25 instantiations of the randomized
factorization algorithm applied to Matrix 2.}
\label{fig:mean_err_matrix2}
\end{figure}

\begin{figure}
\includegraphics[width=\textwidth]{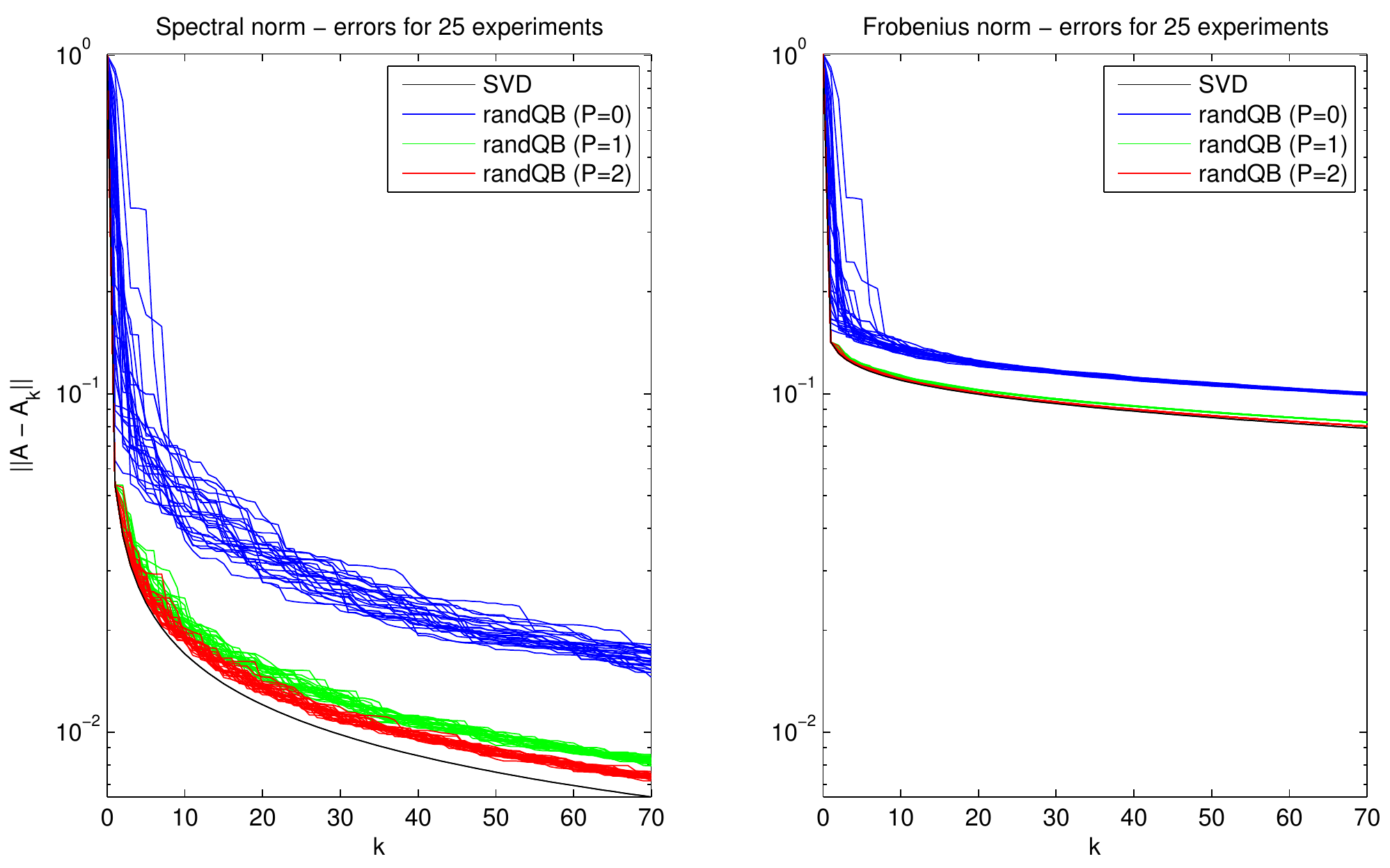}
\caption{The error paths for 25 instantiations of the randomized
factorization algorithm applied to Matrix 2.}
\label{fig:many_err_matrix2}
\end{figure}

\begin{figure}
\includegraphics[width=\textwidth]{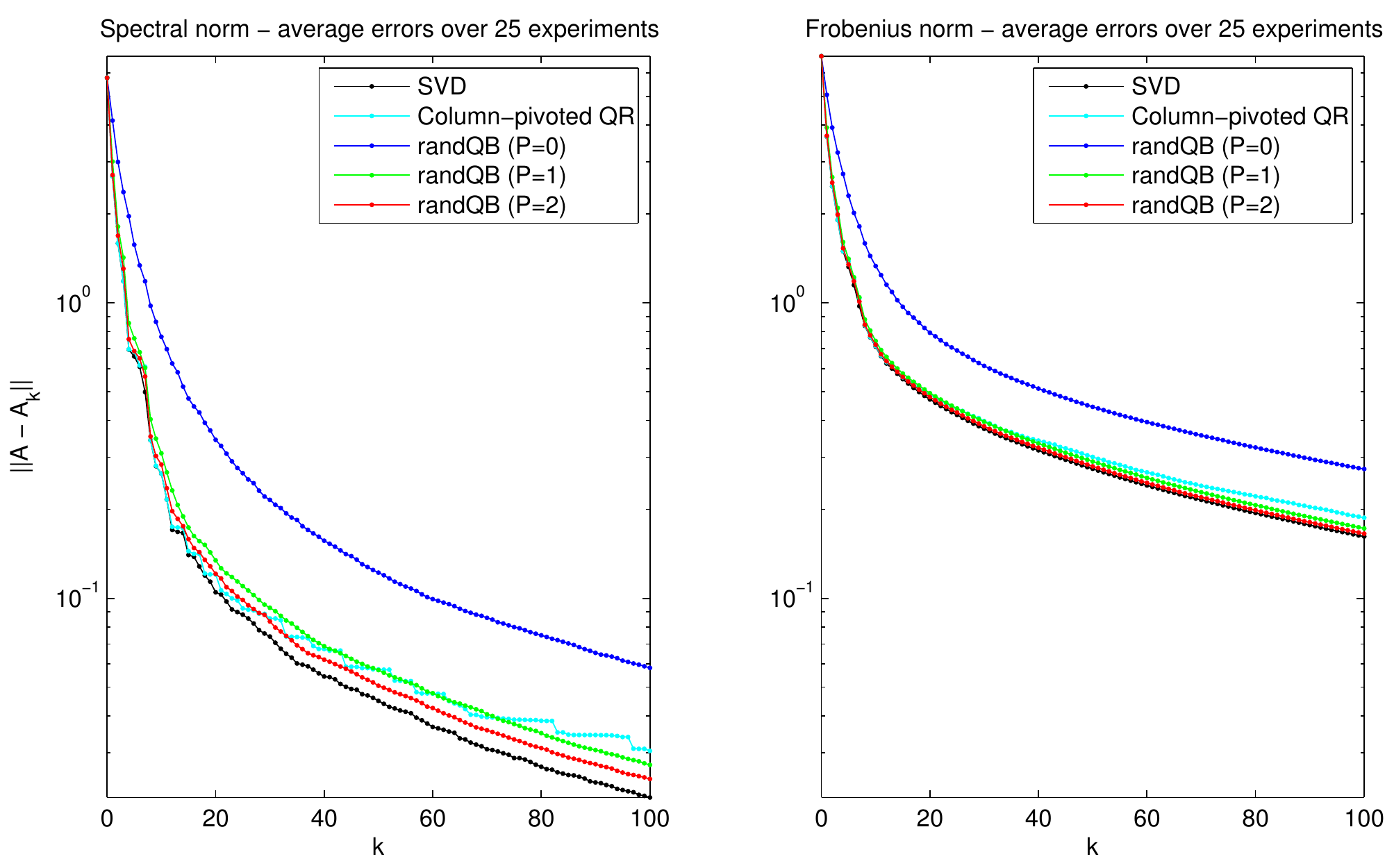}
\caption{The empirical mean errors from 25 instantiations of the randomized
factorization algorithm applied to Matrix 3.}
\label{fig:mean_err_matrix3}
\end{figure}

\begin{figure}
\includegraphics[width=\textwidth]{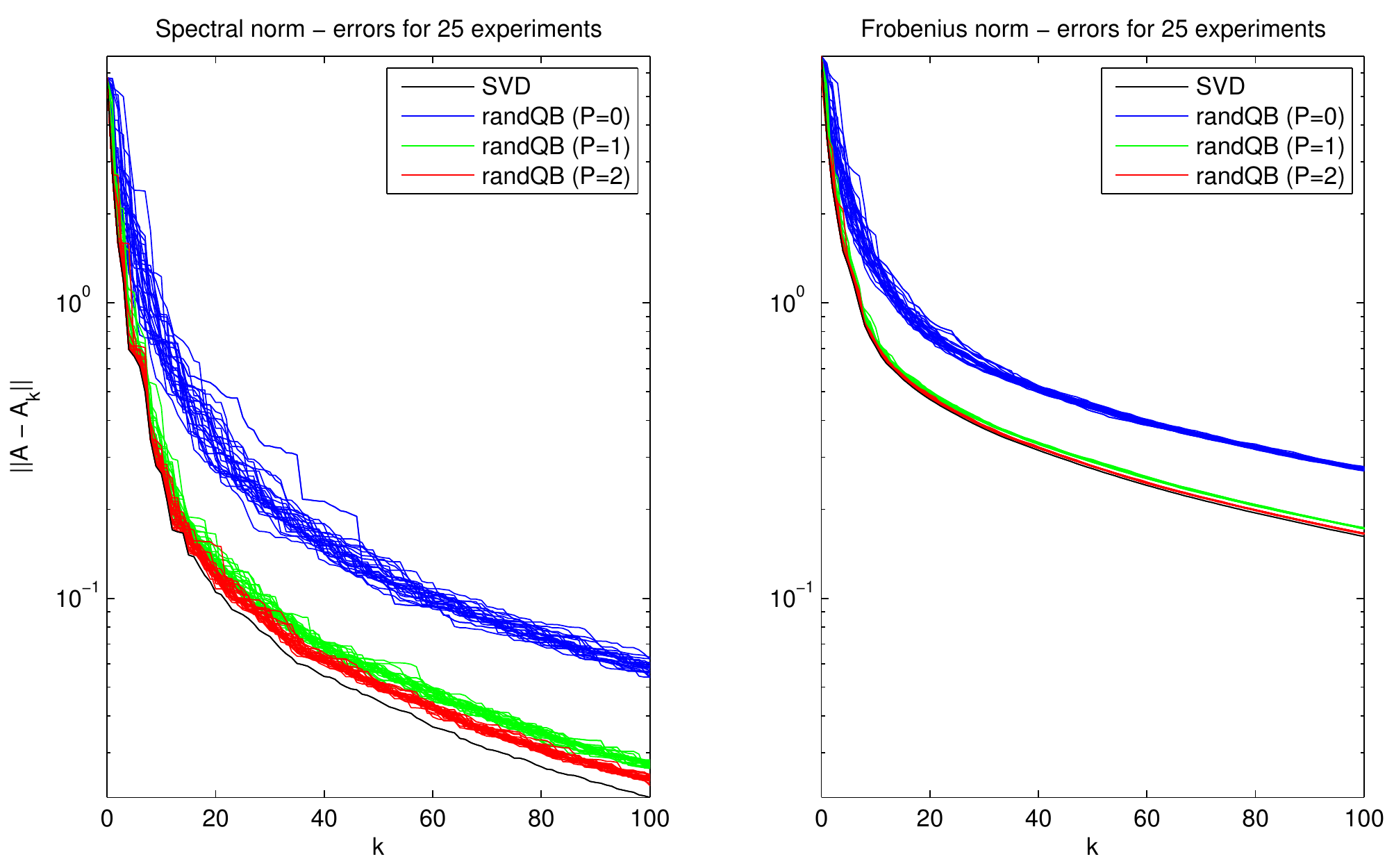}
\caption{The error paths for 25 instantiations of the randomized
factorization algorithm applied to Matrix 3.}
\label{fig:many_err_matrix3}
\end{figure}

\begin{figure}
\includegraphics[width=\textwidth]{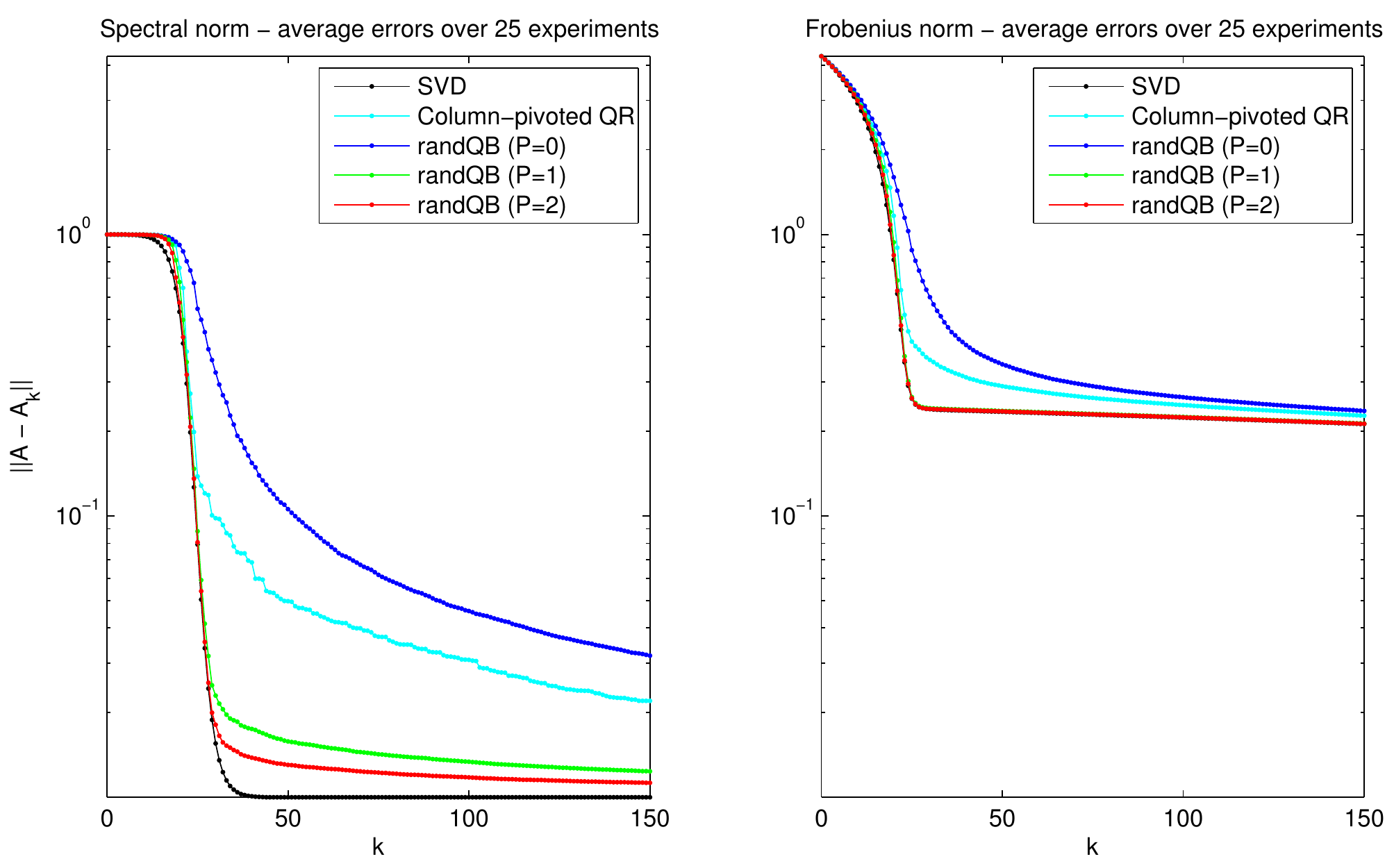}
\caption{The empirical mean errors from 25 instantiations of the randomized
factorization algorithm applied to Matrix 5.}
\label{fig:mean_err_matrix5}
\end{figure}

\begin{figure}
\includegraphics[width=\textwidth]{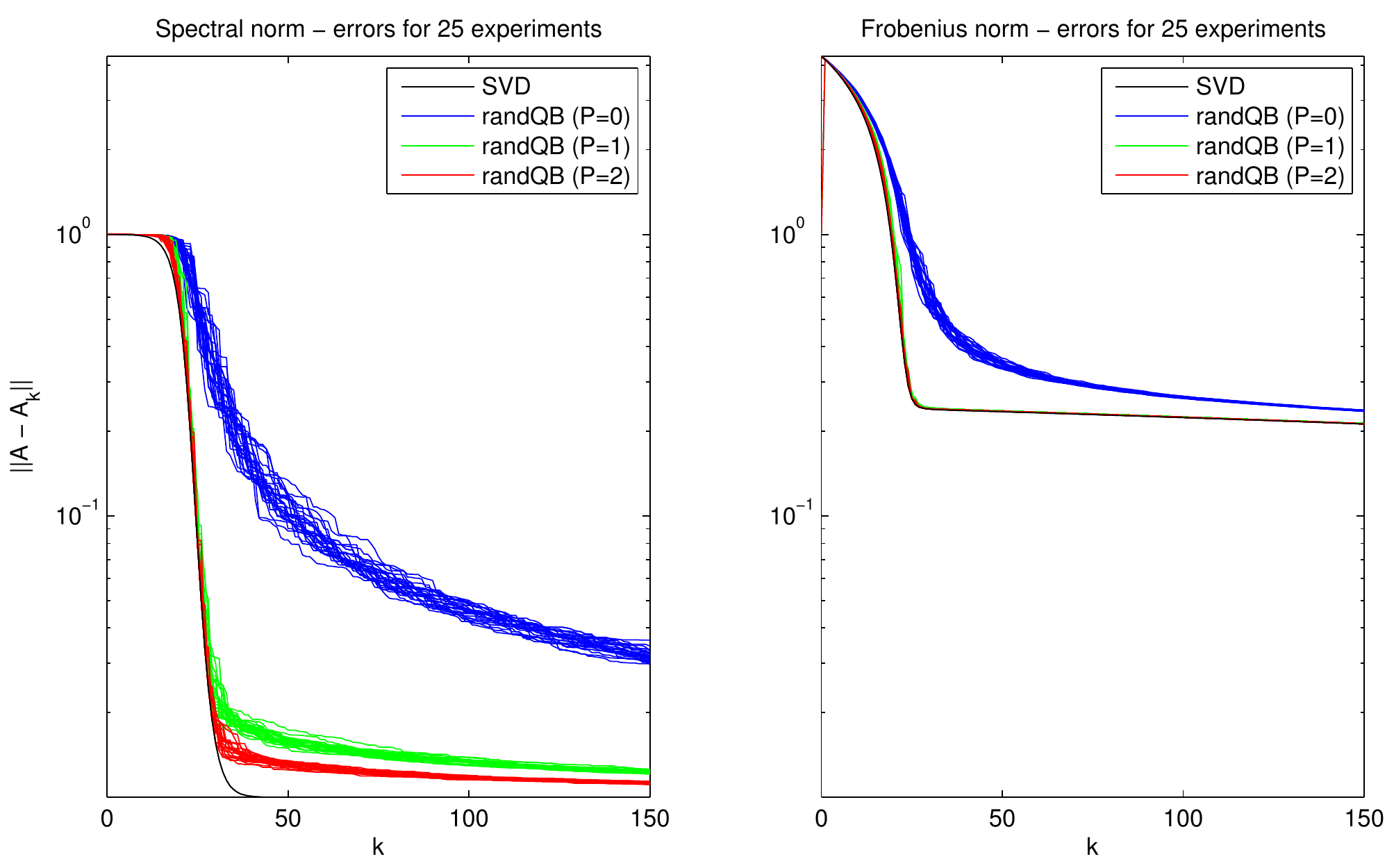}
\caption{The error paths for 25 instantiations of the randomized
factorization algorithm applied to Matrix 5.}
\label{fig:many_err_matrix5}
\end{figure}

\begin{figure}
\includegraphics[width=\textwidth]{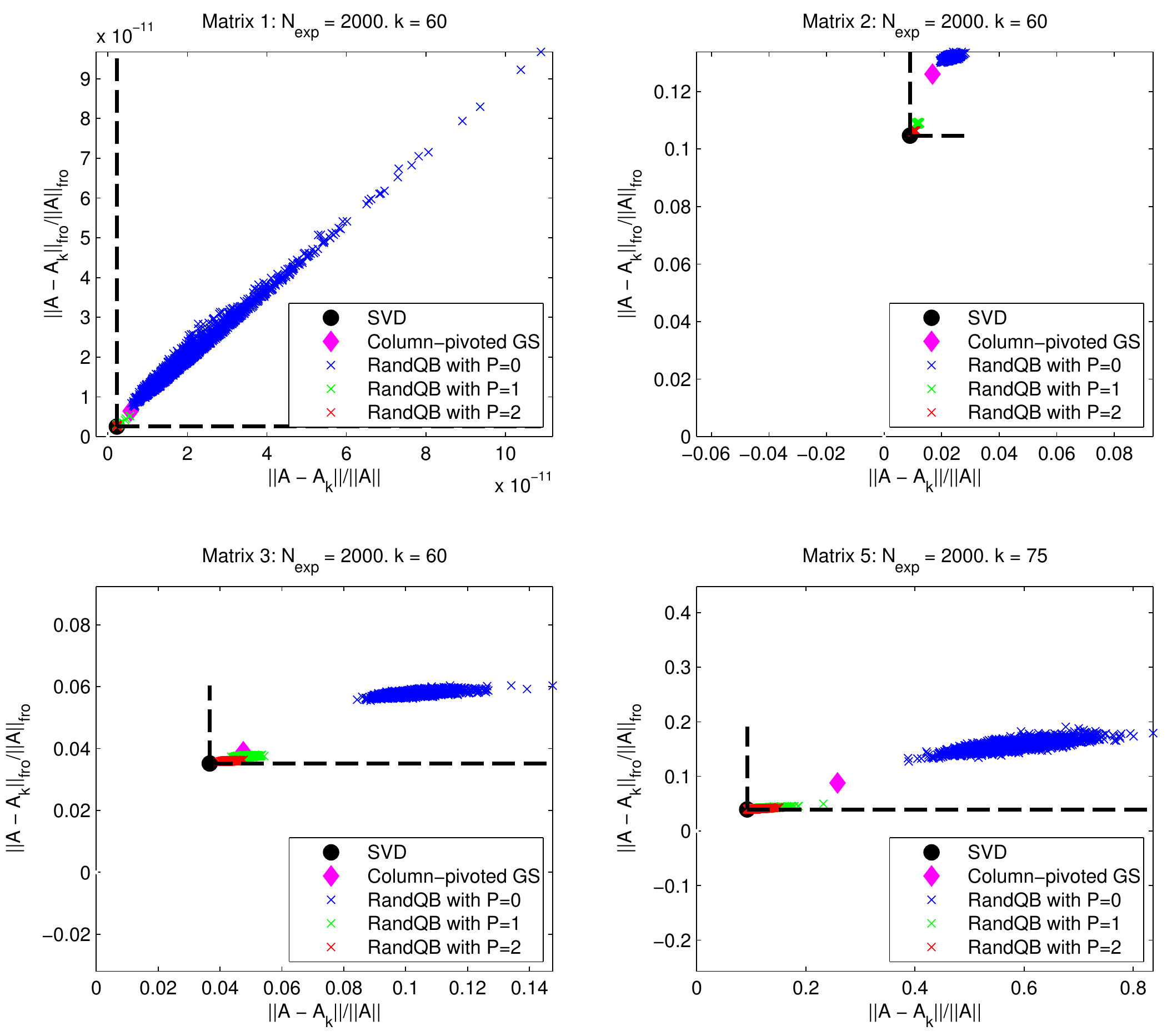}
\caption{Each blue cross in the graphs represents one instantiation
of the randomized blocked algorithm. The $x$- and $y$-coordinates
show the relative errors in the spectral and Frobenius norms, respectively.
For reference, we also include the error from classical column-pivoted
Gram-Schmidt (the magenta diamond), and the error incurred by the truncated
SVD. The dashed lines are the horizonal and vertical lines cutting through
the point representing the SVD --- since these errors are minimal, every
other dot must be located above and to the right of these lines.}
\label{fig:scatter_plots}
\end{figure}

\end{appendix}

\end{document}